\documentclass[a4paper]{amsart}
\usepackage{amsmath,amsthm,amssymb,mathtools,mathrsfs}
\usepackage{bm}
\usepackage{tikz-cd}
\usepackage[all]{xy}
\usepackage{xcolor}
\usepackage{hyperref}
    \definecolor{urlcolor}{rgb}{0,0,0}
    \definecolor{linkcolor}{rgb}{.7,0.10,0.2}
    \definecolor{citecolor}{rgb}{.12,.54,.11}
\hypersetup{
      breaklinks=true, 
      colorlinks=true,
      urlcolor=urlcolor,
      linkcolor=linkcolor,
      citecolor=citecolor,
}

\usepackage{setspace}
\usepackage[shortlabels]{enumitem}
\usepackage[utf8]{inputenc}
\numberwithin{equation}{section}
\usepackage[top=2.25cm,
bottom=2.25cm,
left=2.25cm,
right=2.25cm]{geometry}

\usepackage[
style=alphabetic,
doi=false,
isbn=false,
url=false,
eprint=false,
block = space,
maxnames=50,
sorting=anyt]{biblatex} 

\newtheorem{theorem}{Theorem}[section]
\newtheorem{corollary}[theorem]{Corollary}
\newtheorem{proposition}[theorem]{Proposition}
\newtheorem{proposition-definition}[theorem]{Proposition-Definition}
\newtheorem{lemma}[theorem]{Lemma}

\newtheorem{conjecture}[theorem]{Conjecture}

\theoremstyle{definition} 
\newtheorem{definition}[theorem]{Definition}

\newtheorem{theorem-definition}[theorem]{Theorem-Definition}

\theoremstyle{remark} 
\newtheorem{remark}[theorem]{Remark}
\newtheorem{example}[theorem]{Example}

\renewcommand{\implies}{\Rightarrow}

\renewcommand{\geq}{\geqslant}
\renewcommand{\leq}{\leqslant}
\renewcommand{\subset}{\subseteq}
\renewcommand{\setminus}{\smallsetminus}
\renewcommand{\tilde}{\widetilde}

\newcommand{\BD}{\mathbb{D}}

\newcommand{\BoC}{\mathbf{C}}

\newcommand{\BoG}{\mathbf{G}}

\newcommand{\BoN}{\mathbf{N}}
\newcommand{\BoQ}{\mathbf{Q}}

\newcommand{\BoZ}{\mathbf{Z}}

\newcommand{\CA}{\mathcal{A}}

\newcommand{\CD}{\mathcal{D}}

\newcommand{\CH}{\mathcal{H}}

\newcommand{\CL}{\mathcal{L}}

\newcommand{\CO}{\mathcal{O}}
\newcommand{\CP}{\mathcal{P}}

\newcommand{\CW}{\mathcal{W}}
\newcommand{\CX}{\mathcal{X}}

\newcommand{\FS}{\mathfrak{S}}

\newcommand{\FX}{\mathfrak{X}}

\newcommand{\SL}{\mathscr{L}}

\newcommand{\SP}{\mathscr{P}}

\newcommand{\rmc}{\mathrm{c}}
\newcommand{\cms}{/\!\!/}
\newcommand{\Fg}{\mathfrak{g}}
\newcommand{\GL}{\mathrm{GL}}
\newcommand{\Fh}{\mathfrak{h}}
\newcommand{\Frac}{\mathrm{Frac}}
\newcommand{\HO}{\mathrm{H}}
\newcommand{\Hom}{\mathrm{Hom}}
\newcommand{\IC}{\mathcal{IC}}
\newcommand{\id}{\mathrm{id}}
\newcommand{\IH}{\mathrm{IH}}

\newcommand{\Ind}{\mathrm{Ind}}
\newcommand{\Lie}{\mathrm{Lie}}
\newcommand{\loc}{\mathrm{loc}}
\newcommand{\rmm}{\mathrm{m}}
\newcommand{\prim}{\mathrm{prim}}
\newcommand{\pt}{\mathrm{pt}}
\newcommand{\Res}{\mathrm{Res}}
\newcommand{\RHom}{\mathrm{RHom}}

\newcommand{\sgn}{\mathrm{sgn}}

\newcommand{\Spec}{\mathrm{Spec}}

\newcommand{\Sym}{\mathrm{Sym}}
\newcommand{\rmSL}{\mathrm{SL}}
\newcommand{\Ft}{\mathfrak{t}}
\newcommand{\Tan}{\mathrm{T}}
\newcommand{\vir}{\mathrm{vir}}

\newcommand{\dd}{\mathbf{d}}

\title{Cohomological integrality for weakly symmetric representations of reductive groups}
\date{\today}

\author{Lucien Hennecart}

\address{Laboratoire Ami\'enois de Math\'ematique Fondamentale et Appliqu\'ee, CNRS UMR 7352, Universit\'e de Picardie Jules Verne, 33 rue Saint Leu, 80000 Amiens, France}\email{lucien.hennecart@u-picardie.fr}


\begin{document}

\begin{abstract}
In this paper, we prove the integrality conjecture for quotient stacks arising from weakly symmetric representations of reductive groups. Our main result is a decomposition of the cohomology of the stack into finite-dimensional components indexed by some equivalence classes of cocharacters of a maximal torus. This decomposition enables the definition of new enumerative invariants associated with the stack, which we begin to explore.
\end{abstract}

\maketitle

\setcounter{tocdepth}{1}
\tableofcontents
\section{Introduction}

\subsection{Invariant Theory}

Given a connected reductive group $G$ and a finite-dimensional representation $V$ of $G$, the quotient stack $\FX=V/G$ and the Geometric Invariant Theory (GIT) quotient $\CX=V\cms G$ are fundamental objects in invariant theory \cite{mumford1994geometric, dieudonne1970invariant, popov1994invariant}, with significant applications to geometry and representation theory, e.g. \cite{nakajima1994instantons, reineke2008moduli,maulik2019quantum} for connections with Lie theory.

Recent research has continued to explore the geometry of the quotient $V/G$. For instance, in \cite{vspenko2017non, vspenko2021semi}, the authors study semiorthogonal decompositions of the coherent derived category $\CD(V/G)$ to obtain noncommutative resolutions of $V\cms G$ \cite{van2004non,lunts2010categorical,kuznetsov2015categorical}, using the framework of window categories developed by Halpern-Leistner \cite{halpern2015derived}. Their approach focuses on \emph{quasi-symmetric} representations of reductive groups, which are the representations for which the sum of weights on any line through the origin in the character lattice is zero.

In this paper, we consider the case where $V$ is a \emph{weakly symmetric} representation of $G$ (Definition~\ref{definition:symmetricrepresentation}). This encompasses the case of self-dual representations. We prove a decomposition of the (infinite-dimensional) cohomology vector space $\HO^*(V/G)$ into a finite collection of finite-dimensional, cohomologically graded vector spaces $\CP_{\lambda}$ indexed by cocharacters $\lambda\in X_*(T)$ of a maximal torus $T$ of $G$ (Theorem~\ref{theorem:cohintabsolute}). This decomposition, which employs parabolic induction (\S\ref{section:parabolicinductionrep}), is known as \emph{cohomological integrality} and can be seen as a cohomological version of a generalised Springer decomposition \cite{lusztig1984intersection} in this context. We then define the refined Donaldson--Thomas invariants of $V/G$ as the Betti numbers of these vector spaces $\CP_{\lambda}$ (Definition~\ref{definition:refinedDTinvariants}) following \cite{meinhardt2019donaldson,meinhardt2015donaldson,davison2020cohomological} and initiate their study (Proposition~\ref{proposition:bounds}). These invariants are expected to encode rich representation-theoretic information about the action of $G$ on $V$ and topological information about the GIT quotient $V\cms G$. In particular, we expect that the vector space $\CP_{\lambda}$ computes the intersection cohomology of some GIT quotient, thereby providing an efficient algorithmic and inductive way to determine the intersection Betti numbers of a large class of GIT quotients. We study this question in \cite{hennecart2024cohomological2} and in forthcoming work.

Stacks of the form $V/G$ and their moduli spaces $V\cms G$ serve as local models for smooth stacks $\FX$ with good moduli spaces $\CX$ by the Luna étale slice theorem \cite{alper2020luna}. Therefore, our results provide a foundation for exploring the geometry and topology of morphisms $\FX\rightarrow\CX$ from a weakly symmetric smooth stack to its good moduli space \cite{alper2013good}. While this paper develops the core algebraic results, the applications to the geometry of a class of significant stacks in algebraic geometry and representation theory is investigated in \cite{hennecart2024cohomological2}.

\subsection{Cohomological Integrality}

Cohomological integrality results, originally developed by Kontsevich and Soibelman in the context of cohomological Hall algebras and wall-crossing for moduli spaces \cite{kontsevich2008stability, kontsevich2011cohomological}, involve factorizations of generating series into plethystic exponentials. For \emph{symmetric} quivers without potential, Efimov \cite{efimov2012cohomological} proved a cohomological integrality isomorphism, answering a conjecture by Kontsevich and Soibelman \cite{kontsevich2011cohomological}, which identifies the cohomological Hall algebra (CoHA) to a supercommutative algebra over a graded vector space with finite-dimensional graded components. Efimov's isomorphism can be seen as a categorification of integrality formulas, in the sense that an equality between two generating series is turned into an isomorphism between two graded vector spaces whose graded characters recover the series.

Efimov's integrality theorem forms the basis of all subsequent cohomological integrality results for CoHAs \cite{meinhardt2019donaldson, davison2020cohomological, kapranov2022cohomological, davison2022bps, davison2023bps}. These results provide deep structural insights and link the topology of the stack to the good moduli space. In this context, cohomological integrality is often expressed through Poincaré-Birkhoff-Witt (PBW) theorems, analogous to those in the theory of Lie algebras, and leads to refined cohomological invariants of interest in algebraic geometry, forming the basis of \emph{cohomological Donaldson--Thomas theory} \cite{szendroi2014cohomological, davison2020cohomological, davison2023boson}.

In contrast to the traditional approach involving cohomological Hall algebras, our work deals directly with parabolic induction morphisms (\S\ref{section:parabolicinductionrep}) for more general smooth weakly symmetric stacks arising from weakly symmetric representations of reductive groups. We define the cohomological integrality morphism explicitly in Theorem~\ref{theorem:cohintabsolute}. It does not involve the symmetric power of a vector space as in the quiver case, but some combinatorial replacement of it for any Weyl group, defined from the lattice of cocharacter of a maximal torus.

This paper can be viewed as a generalization of Efimov's integrality isomorphism \cite{efimov2012cohomological} to all weakly symmetric stacks of the form $V/G$, not just those arising from quivers. It lays the groundwork for extending cohomological Donaldson--Thomas theory to more general stacks and moduli spaces beyond those parameterizing objects in an Abelian category. The cohomological integrality we present here is a purely algebraic result, with a purely algebraic proof. We explain in \cite{hennecart2024cohomological2} the geometric implications of this result.

There are also motivic versions of integrality results, such as those studied in \cite{bu2024motivic} following \cite{kontsevich2008stability} and in the context of motivic Donaldson--Thomas theory \cite{kontsevich2010motivic, meinhardt2017introduction, bussi2019motivic}, which are different from ours.

\subsection{Main results}
In this paper, we work over an algebraically closed field of characteristic zero. Moreover, cohomology is considered with $\BoQ$-coefficients.

\subsubsection{Parabolic induction}
We refer to \S\ref{subsection:inductiondiagram} for the precise definitions. Let $G$ be a connected reductive group and $V$ a representation of $G$. We let $T\subset G$ be a maximal torus. For any cocharacter $\lambda\colon\BoG_{\rmm}\rightarrow T$, we let $G^{\lambda}\subset G$ be the corresponding Levi subgroup of $G$ \cite[Proposition 1.22]{digne2020representations} and $V^{\lambda}$ the representation of $G^{\lambda}$ given by the $\lambda$-fixed subspace of $V$. If $\lambda=0$ is the trivial cocharacter, we write $V=V^0$ and $G=G^0$.

There is an induction map
\begin{equation}
\label{equation:inductionintroduction}
 \Ind_{\lambda}\colon \HO^{*+d_{\lambda}}(V^{\lambda}/G^{\lambda})\rightarrow\HO^{*+d_{0}}(V/G)\,,
\end{equation}
where $d_{\lambda}=\dim V^{\lambda}-\dim G^{\lambda}$ (\S\ref{subsection:parabolicinduction}). If $V$ is symmetric (i.e. $V$ and $V^*$ have the same sets of weights, Definition~\ref{definition:symmetricrepresentation}), then the induction map preserves shifted cohomological degrees (Lemma~\ref{lemma:preservescohdegrees}), that is $\Ind_{\lambda}(\HO^{i+d_{\lambda}}(V^{\lambda}/G^{\lambda}))\subset\HO^{i+d_0}(V/G)$ for any $i\in\BoZ$ (this is the reason for the notation $\HO^{*+d_{\lambda}}$ in \eqref{equation:inductionintroduction}, which emphasizes that the natural grading is the cohomological grading shifted by $d_{\lambda}$).

We define an order on the set $X_*(T)$ of cocharacters of $T$ as follows:
\begin{equation}
\label{equation:orderintro}
 \lambda\preceq\mu \stackrel{\text{def.}}{\iff}
 \left\{
 \begin{aligned}
  &V^{\lambda}\subset V^{\mu}\\
& \mathfrak{g}^{\lambda}\subset\mathfrak{g}^{\mu}\,.
 \end{aligned}
\right.
\end{equation}

We let $\SP_{V}\coloneqq X_*(T)/\sim$, where $\sim$ is the equivalence relation given by $\lambda\sim\mu\iff (\lambda\preceq\mu \text{ and }\mu\preceq\lambda)$. In other words, two cocharacters are equivalent if their fixed-point sets  inside $V$ and $\Fg$ (or equivalently inside $V$ and $G$) both coincide. The set $\SP_V$ is finite. If $V$ is symmetric, the induction map $\Ind_{\lambda}$ only depends on the class $\overline{\lambda}\in\SP_V$ up to a non-zero scalar (Lemma~\ref{lemma:sameimageoverline}).

 We let $W=N_G(T)/T$ be the Weyl group of $G$. The natural action of $W$ on $X_*(T)$ descends to a $W$-action on $\SP_V$ (\S\ref{subsection:parabolicinduction}). The image of the induction map $\Ind_{\lambda}$ only depends on the class $\tilde{\lambda}\in\SP_V/W$ of $\lambda$ (Lemma~\ref{lemma:sameimageWeylgroup}).

\subsubsection{Cohomological integrality}
Let $G$ be a reductive group with maximal torus $T$ and $V$ be a finite dimensional \emph{weakly symmetric} representation of $G$ (Definition~\ref{definition:symmetricrepresentation}).
For any $\overline{\lambda}\in\SP_V$, we let $\lambda\in X_*(T)$ be an arbitrary lift. For $\tilde{\lambda}\in\SP_V/W$, we let $\lambda\in X_*(T)$ be an arbitrary lift.

For $\lambda\in X_*(T)$, we let $W_{\lambda}=\{w\in W\mid V^{w\cdot\lambda}=V^{\lambda}\text{ and }\Fg^{w\cdot\lambda}=\Fg^{\lambda}\}$. This is a subgroup of $W$ (Lemma~\ref{lemma:Wlambdasubgroup}).

We let $G_{\lambda}\subset G$ be the intersection of the kernel of the action map $G^{\lambda}\rightarrow\GL(V^{\lambda})$ with the center of $G^{\lambda}$ and $\Fg_{\lambda}\coloneqq\Lie(G_{\lambda})$. Note that the group $G_{\lambda}$ is not necessarily connected, but this can essentially be ignored (Lemma~\ref{lemma:G_lambdacohomology}) because we consider cohomology with $\BoQ$-coefficients.

For $\alpha\in X^*(T)$, we let $V_{\alpha}\coloneqq \{v\in V\mid \forall t\in T, t\cdot v=\alpha(t)v\}$ be the $\alpha$-weight space of $V$. We define $\Fg_{\alpha}$ similarly.

We let
\[
 k_{\lambda}\coloneqq \frac{\prod_{\alpha\in X^*(T), \langle\lambda,\alpha\rangle<0}\alpha^{\dim V_{\alpha}}}{\prod_{\alpha\in X^*(T), \langle\lambda,\alpha\rangle<0}\alpha^{\dim\Fg_{\alpha}}}\,.
\]
We may call $k_{\lambda}$ the \emph{induction kernel}, see the shuffle-like formula in Proposition~\ref{proposition:explicitformulainduction}. We define $\varepsilon_{V,\lambda}\colon W_{\lambda}\rightarrow\BoC^*$ to be the character such that for any $w\in W_{\lambda}$, $w(k_{\lambda})=\varepsilon_{V,\lambda}(w)^{-1}k_{\lambda}$ (Proposition~\ref{proposition:characterW_lambda}).

\begin{theorem}[Proposition~\ref{proposition:finitedimensional} + Corollary~\ref{corollary:cohintiso}]
\label{theorem:cohintabsolute}
 There exist $W_{\lambda}$-stable cohomologically graded vector spaces $\CP_{\lambda}\subset \HO^{*+d_{\lambda}}(V^{\lambda}/G^{\lambda})$, $\lambda\in X_*(T)$, such that the map
 \[
  \bigoplus_{\tilde{\lambda}\in\SP_{V}/W}(\CP_{\lambda}\otimes \HO^*(\pt/G_{\lambda}))^{\varepsilon_{V,\lambda}}\xrightarrow{\sum_{\tilde{\lambda}\in\SP_V/W}\Ind_{\lambda}} \HO^*(V/G)
 \]
 induced by the induction morphisms \eqref{equation:inductionintroduction} is a graded isomorphism, where $\HO^*(\pt/G_{\lambda})\cong \Sym(\Fg_{\lambda}^*)$, $\Fg_{\lambda}^*$ sits in degree $2$ and $(\CP_{\lambda}\otimes \HO^*(\pt/G_{\lambda}))^{\varepsilon_{V,\lambda}}$ denotes the $\varepsilon_{V,\lambda}$-isotypic component for the natural $W_{\lambda}$-action. Moreover, the vector spaces $\CP_{\lambda}$ are finite-dimensional and the inclusion $\CP_{\lambda}\subset \HO^{*+d_{\lambda}}(V^{\lambda}/G^{\lambda})$ factors through the inclusion $\HO^{*+d_{\lambda}}(V^{\lambda}/(G^{\lambda}/G_{\lambda}))\rightarrow\HO^{*+d_{\lambda}}(V^{\lambda}/G^{\lambda})$. In addition, the vector spaces $\CP_{\mu}$ for $\mu\preceq\lambda$ satisfy the cohomological integrality isomorphism for $(V^{\lambda},G^{\lambda})$.
\end{theorem}
This type of theorems is sometimes referred to as \emph{cohomological integrality}. We can reconstruct the infinite dimensional vector space $\HO^*(V/G)$ using a finite number of finite-dimensional vector spaces $\CP_{\lambda}$, $\tilde{\lambda}\in\SP_V/W$, allowing us to extract fundamental enumerative invariants associated to the stack $V/G$ (Definition~\ref{definition:refinedDTinvariants}).

For a cocharacter $\lambda\in X_*(T)$, there is a natural $W$-equivariant map $\SP_{V^{\lambda}}\rightarrow\SP_V$, which can be defined as follows. Given a cocharacter $\mu\in X_*(T)$, there exists a cocharacter $\nu\in X_*(T)$ such that $V^{\lambda}\cap V^{\mu}=V^{\nu}$ and $G^{\lambda}\cap G^{\mu}=G^{\nu}$ (Lemma~\ref{lemma:gcdcocharacters}). We have $\nu\preceq \lambda$ and moreover, $\overline{\nu}=\overline{\mu}$ in $\SP_{V^{\lambda}}$. The map sends $\overline{\mu}=\overline{\nu}\in\SP_{V^{\lambda}}$ to $\overline{\nu}\in\SP_V$. Then, the last sentence of Theorem~\ref{theorem:cohintabsolute} means that the vector spaces $\CP_{\nu}$ satisfy the cohomological integrality isomorphism for $(V^{\lambda},G^{\lambda})$.

\begin{remark}
 Let $Q=(Q_0,Q_1)$ be a symmetric quiver with set of vertices $Q_0$ and set of arrows $Q_1$. Let $\dd\in\BoN^{Q_0}$ be a dimension vector. If $V_{\dd}$ is the representation space of $\dd$-dimensional representations of $Q$ which is acted on by $\GL_{\dd}\coloneqq\prod_{i\in Q_0}\GL_{\dd_i}$, then the isomorphism given by  Theorem~\ref{theorem:cohintabsolute} essentially recovers \cite[Theorem 1.1]{efimov2012cohomological}. In \emph{loc.cit}, it was convenient for the author to twist the cohomological Hall algebra multiplication to make it supercommutative, as in \cite{kontsevich2011cohomological}. This twist does not change the images of the induction/multiplication morphisms which means that it is possible to have cohomological integrality isomorphisms even without this twist. This remark inspired the present work, in which the twist is replaced by the character $\varepsilon_{V,\lambda}$ of the group $W_{\lambda}$ (Proposition~\ref{proposition:characterW_lambda}). When $V$ is \emph{symmetric} (and not only \emph{weakly symmetric}), the character $\varepsilon_{V,\lambda}$ takes values in $\{\pm1\}$.
\end{remark}

\begin{remark}
 In this paper, we consider cohomology with $\BoQ$-coefficients, which is crucial for the cohomological integrality map (Theorem~\ref{theorem:cohintabsolute}) to be an isomorphism. It is crucial to work with $\BoQ$-coefficients from \S\ref{subsection:explicitformula} until the end of the paper. However, one can define the induction morphisms (\S\ref{section:parabolicinductionrep}) for any ring of coefficients. It then becomes an interesting question to study the analogous question of cohomological integrality in positive characteristic, especially when the cohomology of the classifying space $\HO^*(\pt/G)$ has torsion.
\end{remark}

\begin{definition}
\label{definition:refinedDTinvariants}
We define the \emph{cohomologically refined Donaldson--Thomas invariants} of $V/G$ as the dimensions $p_{\lambda,i}\coloneqq\dim \CP_{\lambda}^i\subset\HO^{i+d_{\lambda}}(V^{\lambda}/G^{\lambda})$ of the cohomological degree $i$ piece of the vector space $\CP_{\lambda}$ for $\lambda\in X_*(T)$. We may call the Euler characteristics $p_{\lambda}\coloneqq \sum_{i\in\BoZ}(-1)^ip_{\lambda,i}$ of $\CP_{\lambda}$ the \emph{Donaldson--Thomas invariants of $V/G$}.
\end{definition}

\subsubsection{Vanishing bound}
The vector spaces $\CP_{\lambda}$ given by Theorem~\ref{theorem:cohintabsolute} are finite dimensional (Proposition~\ref{proposition:finitedimensional}), and we can give a rough bound on the cohomological degrees for which they possibly do not vanish. We let $\CH_{\lambda}\coloneqq\HO^{*+d_{\lambda}}(V^{\lambda}/G^{\lambda})$, where $d_{\lambda}=\dim V^{\lambda}-\dim G^{\lambda}$. The natural grading on $\CH_{\lambda}$ is given by $\CH_{\lambda}^k=\HO^{k+d_{\lambda}}(V^{\lambda}/G^{\lambda})$. It induces a cohomological grading on $\CP_{\lambda}$.

\begin{proposition}
\label{proposition:bounds}
 The cohomologically graded vector space $\CP_{\lambda}$ given by Theorem~\ref{theorem:cohintabsolute} are concentrated in cohomological degrees $[\dim G^{\lambda}-\dim V^{\lambda}, \dim G^{\lambda}]$, that is $\CP_{\lambda}^k=0$ unless $k\in [\dim G^{\lambda}-\dim V^{\lambda}, \dim G^{\lambda}]$.
\end{proposition}

\subsection{Notations and conventions}
\label{subsection:notations}
\begin{enumerate}[$\bullet$]
 \item If $V$ is a representation of a finite group $W$ and $\chi$ an irreducible character of $W$, we let $V^{\chi}$ be the $\chi$-isotypic component of $V$.
 \item If $X$ is a complex algebraic variety acted upon by an algebraic group $G$, we let $X/G=[X/G]$ be the quotient stack, that is we omit the brackets.
 \item The multiplicative group is denoted by $\BoG_{\rmm}$.
 The Abelian group of characters of an algebraic torus $T$ is $\Hom_{\BoZ}(T,\BoG_{\rmm})$. The set of cocharacters of $T$ is the Abelian group $\Hom_{\BoZ}(\BoG_{\rmm},T)$. We have natural embedding $X_*(T)\subset\Ft\coloneqq\Lie(T)$ and $X^*(T)\subset\Ft^*$ respectively given by $\lambda\mapsto \frac{\mathrm{d\lambda}}{\mathrm{d}t}_{|t=1}$ and $\alpha\mapsto \mathrm{d}\alpha(1)$, the differential of a character $\alpha$ at the neutral element $1\in T$. Moreover, these inclusions are compatible with the natural pairings $\langle-,-\rangle\colon X_*(T)\times X^*(T)\rightarrow\BoZ$ and $\Ft\times\Ft^*\rightarrow\BoC$.
 \item If $A$ is an Abelian group and $S\subset A$ is a subset of $A$, we let $-S\coloneqq \{-s\colon s\in A\}$ be the opposite of $S$. We will only consider the case $A=X^*(T)$.
 \item Let $H\subset G$ be algebraic groups and $X$ an $H$-variety. We let $X\times^HG\coloneqq (X\times G)/H$ where $H$ acts on $X\times G$ by $h\cdot (x,g)=(h\cdot x,gh^{-1})$. The formula $g\cdot (x,g')\coloneqq (x,gg')$ gives a $G$-action on $X\times^HG$.
 \item If $R$ is a domain (e.g. $R=\HO^*_T(\pt)$ for some algebraic torus $T$), we denote by $\Frac(R)$ its fraction field.
 \item If $G$ is a possibly disconnected algebraic group, then we denote by $G^{\circ}$ its neutral component.
 \item Given a scheme $X$, we denote by $\CO_X$ its sheaf of regular functions. If $X$ is a complex scheme, then $\BoC[X]\coloneqq \CO_X(X)$ set ring of regular functions on $X$.
 \item If $S$ is a set (for example, a vector space) with the action of a group $G$, we denote by $S^G$ the fixed subset.
\end{enumerate}

\subsection{Acknowledgements}
At various stages of the preparation of this work, the author was supported by the Royal Society, by the National Science Foundation under Grant No. DMS-1928930 and by the Alfred P. Sloan Foundation under grant G-2021-16778, while the author was in residence at the Simons Laufer Mathematical Sciences Institute (formerly MSRI) in Berkeley, California, during the Spring 2024 semester. The author thanks the SLMath Sciences Institute and the University of Edinburgh for the excellent working conditions. The author is grateful to Ben Davison for the support provided, in particular via postdoctoral fellowships, to Woonam Lim for pointing out a mistake in the example section, to Drago\c{s} Fr\u{a}\c{t}il\u{a} for interesting comments, to Sebastian Schlegel Mejia for suggesting improvements and corrections, to Dimitri Wyss for interesting related discussions, and to Andrei Negu\c{t} for suggesting me to investigate weakenings of the symmetricity condition, which led to the current version involving \emph{weakly symmetric} representations (Definition~\ref{definition:symmetricrepresentation}). I am grateful to Daniel Juteau for discussions on perverse sheaves. I would also like to thank the anonymous referee for allowing me to improve this paper.

\section{Parabolic induction for representations of reductive groups}
\label{section:parabolicinductionrep}
\subsection{The induction diagram}
\label{subsection:inductiondiagram}
Let $G$ be a reductive group, $T\subset G$ a maximal torus of $G$ and $V$ a finite-dimensional representation of $G$. For a cocharacter $\lambda\colon\BoG_{\rmm}\rightarrow T$, we let $G^{\lambda}\coloneqq \{g\in G\colon \lambda(t)g\lambda(t)^{-1}=g\}$ be the centraliser of $\lambda$, a Levi subgroup of $G$ \cite[Proposition~1.22]{digne2020representations} and $V^{\lambda}\coloneqq \{v\in V\mid \lambda(t)\cdot v=v\}$ be the fixed locus, a representation of $G^{\lambda}$. We also let $P_{\lambda}= G^{\lambda\geq 0}\coloneqq \{g\in G\mid \lim_{t\rightarrow 0}\limits \lambda(t)g\lambda(t)^{-1}\text{ exists}\}$, a parabolic subgroup of $G$ and $V^{\lambda\geq 0}\coloneqq \{v\in V\mid\lim_{t\rightarrow 0}\limits \lambda(t)\cdot v\text{ exists}\}$, a representation of $P_{\lambda}$.

A weight of $V$ is a character $\alpha\colon T\rightarrow\BoG_{\rmm}$ such that there exists $v\in V\setminus \{0\}$ such that for any $t\in T$, $t\cdot v=\alpha(t)v$.

We let $\CW'(V)\subset X^*(T)$ be the set of weights of the representation $V$. We have a direct sum decomposition
\[
 V\cong\bigoplus_{\alpha\in\CW'(V)}V_{\alpha}
\]
where $V_{\alpha}$ consists of $0$ and nonzero vectors of weight $\alpha$, where $V_{\alpha}\coloneqq \{v\in V\mid t\cdot v=\alpha(t)v \text{ for any $t\in T$}\}$.

We let $\CW(V)$ be the collection of weights of $V$, counted with multiplicities: $\alpha\in\CW'(V)$ appears $\dim V_{\alpha}$ times in $\CW(V)$.

In particular, we have $\CW'(\Fg)$ and $\CW(\Fg)$ where $\Fg$ is seen as the adjoint representation of $G$. We have the natural pairing
\[
 \langle-,-\rangle\colon X_*(T)\times X^*(T)\rightarrow\BoZ
\]
between characters and cocharacters. If $(\lambda,\alpha)\in X_*(T)\times X^*(T)$, then $\alpha\circ\lambda\in\mathrm{End}(\BoG_{\rmm})$, and so it is given by raising to some power, which is $\langle\lambda,\alpha\rangle\in\BoZ$. We let $\CW^{\lambda>0}(V)\coloneqq\{\alpha\in \CW(V)\mid \langle\lambda,\alpha\rangle>0\}$ and we define similarly $\CW^{\lambda=0}(V)$, $\CW^{\lambda\geq0}(V)$, $\CW^{\lambda<0}(V)$ and $\CW^{\lambda\leq0}(V)$ and their $'$ versions, by forgetting the multiplicities.

It is immediate that $V^{\lambda\geq0}=\bigoplus_{\alpha\in\CW^{'\lambda\geq 0}(V)}V_{\alpha}$.

We consider the (commutative) induction diagram
\begin{equation}
\label{equation:comminductiondiagram}
\begin{tikzcd}
	& {V^{\lambda\geq 0}/P_{\lambda}} \\
	{V^{\lambda}/G^{\lambda}} && {V/G} \\
	{V^{\lambda}\cms G^{\lambda}} && {V\cms G}
	\arrow["{q_{\lambda}}"', from=1-2, to=2-1]
	\arrow["{p_{\lambda}}", from=1-2, to=2-3]
	\arrow["{\pi_{\lambda}}"', from=2-1, to=3-1]
	\arrow["\pi", from=2-3, to=3-3]
	\arrow["\imath_{\lambda}",from=3-1, to=3-3]
\end{tikzcd}\end{equation}
where $V\cms G\coloneqq \Spec(\BoC[V]^G)$ is the affine GIT quotient of $V$ by $G$ and the map $\imath_{\lambda}$ is induced by the equivariant closed immersion $(V^{\lambda},G^{\lambda})\rightarrow(V,G)$.

\begin{lemma}
\label{lemma:smoothnessqproperp}
 The map $q_{\lambda}$ is a vector bundle stack of rank $r_{\lambda}\coloneqq\#\CW^{\lambda>0}(V)-\#\CW^{\lambda>0}(\Fg)$ and the map $p_{\lambda}$ is a proper and representable morphism of stacks.
\end{lemma}
\begin{proof}
 The map $p_{\lambda}$ can be presented as the map of quotient stacks $(V^{\lambda\geq 0}\times^{P_{\lambda}} G)/G\rightarrow V/G$ coming from the $G$-equivariant map
 \[
  \tilde{p}_{\lambda}\colon \tilde{V}^{\lambda\geq 0}\coloneqq V^{\lambda\geq 0}\times^{P_{\lambda}}G\rightarrow V, \quad (v,g)\mapsto g\cdot v\,,
 \]
 which can be factored as a closed immersion followed by a projective map
 \[
  \begin{matrix}
   V^{\lambda\geq 0}\times^{P_{\lambda}}G&\rightarrow&V\times G/P_{\lambda}&\rightarrow&V\\
   (v,g)&\mapsto&(g\cdot v,gP_{\lambda})&&\\
   &&(v,gP_{\lambda})&\mapsto&v
  \end{matrix}
 \]
and is therefore projective. Therefore, $p_{\lambda}$ is representable and projective.
 
 The fact that $q_{\lambda}$ is smooth, surjective of relative dimension $\#\CW^{\lambda>0}(V)-\CW^{\lambda>0}(\Fg)$ is easily seen. That it is in addition a vector bundle stack is a rather classical fact (which is not used in this paper). It comes from the facts that $V^{\lambda\geq 0}\rightarrow V^{\lambda}$ is a vector bundle (by the Bia{\l}inicky-Birula decomposition \cite{bialynicki1973some}, but this be seen directly in the case of linear actions) and the kernel of the map $P_{\lambda}\rightarrow G^{\lambda}$ is the unipotent radical of $P_{\lambda}$.
\end{proof}

\begin{lemma}
 The map $\imath_{\lambda}$ is a finite map.
\end{lemma}
\begin{proof}
 The arguments are parallel to that of the proof \cite[Lemma 2.1]{meinhardt2019donaldson}. Namely, we consider the commutative diagram
\[\begin{tikzcd}
	&& {V^{\lambda\geq 0}/P_{\lambda}} \\
	{V^{\lambda}/G^{\lambda}} && {\Spec(\BoC[\tilde{V}^{\lambda\geq 0}]^{G})} && {V/G} \\
	{V^{\lambda}\cms G^{\lambda}} &&&& {V\cms G}
	\arrow["q_{\lambda}"', from=1-3, to=2-1]
	\arrow["\jmath_{\lambda}", bend left=20, from=2-1, to= 1-3]
	\arrow["{\pi_{\lambda\geq 0}}", from=1-3, to=2-3]
	\arrow["p_{\lambda}", from=1-3, to=2-5]
	\arrow["{\pi_{\lambda}}"', from=2-1, to=3-1]
	\arrow["{\overline{q}_{\lambda}}"', from=2-3, to=3-1]
	\arrow["{\overline{p}_{\lambda}}", from=2-3, to=3-5]
	\arrow["\pi", from=2-5, to=3-5]
	\arrow["{\overline{\jmath}_{\lambda}}"', bend right=20, from=3-1, to=2-3]
    \arrow["\imath_{\lambda}"', bend right=20, from=3-1, to=3-5]
    \end{tikzcd}\]
where $\tilde{V}^{\lambda\geq 0}\coloneqq V^{\lambda\geq 0}\times^{P_{\lambda}}G$ (see \S\ref{subsection:notations}), so that we have an equivalence of stacks $V^{\lambda\geq 0}/P_{\lambda}\simeq\tilde{V}^{\lambda\geq 0}/G$. We now explain the definitions of the maps in this diagram.

The map $q_{\lambda}$ is given by the equivariant map $(V^{\lambda\geq 0},P_{\lambda})\rightarrow (V^{\lambda},G^{\lambda})$, $(v,g)\mapsto(\lim_{t\rightarrow0}\limits \lambda(t)\cdot v,\lim_{t\rightarrow 0}\limits\lambda(t)g\lambda(t)^{-1})$. Therefore, the pullback of functions by $q_{\lambda}$ induces a morphism between rings of invariant functions
\[
 \overline{q}_{\lambda}^*\colon\BoC[V^{\lambda}]^{G^{\lambda}}\rightarrow \BoC[\tilde{V}^{\lambda\geq 0}]^{G}
\]
whose dual is the morphism between affine schemes
\[
\overline{q}_{\lambda}\colon \Spec(\BoC[\tilde{V}^{\lambda\geq 0}]^{G})\rightarrow\Spec(\BoC[V^{\lambda}]^{G^{\lambda}})=V^{\lambda}\cms G^{\lambda}\,.
\]
The map $\jmath_{\lambda}$ is given by the equivariant morphism $(V^{\lambda},G^{\lambda})\rightarrow(V^{\lambda\geq 0},P_{\lambda})$, $(v,g)\mapsto (v,g)$ which induces the morphism of algebras
\[
 \overline{\jmath}_{\lambda}^*\colon \BoC[\tilde{V}^{\lambda\geq 0}]^{G}\rightarrow\BoC[V^{\lambda}]^{G^{\lambda}}
\]
and dually the morphism of schemes
\[
 \overline{\jmath}_{\lambda}\colon\Spec(\BoC[V^{\lambda}]^{G^{\lambda}})\rightarrow\Spec(\BoC[\tilde{V}^{\lambda\geq 0}]^{G})\,.
\]

Since $q_{\lambda}\circ \jmath_{\lambda}=\id$, we have $\overline{q}_{\lambda}\circ\overline{\jmath}_{\lambda}=\id$. On the ring of functions, we have $\overline{\jmath}_{\lambda}^*\circ \overline{q}_{\lambda}^*=\id$ and so $\overline{\jmath}_{\lambda}^*$ is surjective. This tells us that $\overline{\jmath}_{\lambda}$ is a closed immersion. Moreover, $\imath_{\lambda}=\overline{p}_{\lambda}\circ\overline{\jmath}_{\lambda}$ and so it suffices to prove that $\overline{p}_{\lambda}$ is a finite map. Since it is a morphism of finite type complex schemes, it suffices to prove that it is an integral morphism \cite[\href{https://stacks.math.columbia.edu/tag/01WJ}{Lemma 01WJ}]{stacks-project}.

Since $\tilde{p}_{\lambda}\colon \tilde{V}^{\lambda\geq 0}\rightarrow V$ is a projective morphism (Proof of Lemma~\ref{lemma:smoothnessqproperp}), $(\tilde{p}_{\lambda})_*\CO_{\tilde{V}^{\lambda\geq 0}}$ is a coherent $\CO_{V}$-module. Therefore, the map $\BoC[V]\rightarrow \BoC[\tilde{V}^{\lambda\geq 0}]$ is finite and hence integral  \cite[\href{https://stacks.math.columbia.edu/tag/01WJ}{Lemma 01WJ}]{stacks-project}. We deduce that the map
\[
 \overline{p}_{\lambda}^*\colon\BoC[V]^G\rightarrow \BoC[\tilde{V}^{\lambda\geq 0}]^G
\]
is integral, and hence finite, as follows. Let $a\in\BoC[\tilde{V}^{\geq\lambda}]^G$. It is solution of a monic polynomial with coefficients in $\BoC[V]$. By applying the Reynolds operator to the equation, we obtain a monic equation with coefficients in $\BoC[V]^G$ of which $a$ is a zero. This concludes the proof.

\end{proof}

We define the equivalence relation $\equiv$ on $X_*(T)$ by $\alpha\equiv\beta\iff \exists a,b\in\BoZ_{>0},\quad a\alpha=b\beta$. The quotient set is denoted by $\overline{X_*(T)}\coloneqq X_*(T)/\equiv$. If $S\subset X_*(T)$, we denote by $\overline{S}$ its image in $\overline{X_*(T)}$. Similarly, if $S$ is a collection of characters (with multiplicities allowed), we denote by $\overline{S}$ the corresponding collection of elements of $\overline{X_*(T)}$ (with multiplicities taken into account).

\begin{definition}
\label{definition:symmetricrepresentation}
We say that a representation $V$ of a reductive group $G$ is
\begin{enumerate}
 \item  \emph{symmetric} if $V$ and its dual $V^*$ have the same sets of weights counted with multiplicities: $\CW(V)=\CW(V^*)$.
 \item \emph{weakly symmetric} if $\overline{\CW(V)}=\overline{\CW(V^*)}$.
\end{enumerate}
\end{definition}
Since a representation of a reductive group is determined by its character, a representation $V$ of $G$ is symmetric if and only if it is self-dual, that is $V\cong V^*$ as representations of $G$.
\begin{example}
\label{example:symmetricrepresentations}
The following are symmetric representations:
 \begin{enumerate}
  \item The trivial representation $\{0\}$ of a reductive group $G$, 
  \item The adjoint representation $\Fg=\Lie(G)$ of a reductive group $G$,
  \item The cotangent representation $\Tan^*V$ of any representation $V$ of a reductive group $G$,
  \item The representation space of $\dd\in\BoC^{Q_0}$-dimensional representations of a \emph{symmetric}  quiver $Q=(Q_0,Q_1)$, for $G=\prod_{i\in Q_0}\GL_{\dd_i}$.
  \item Any representations of products of simple groups of types $B_n, C_n, D_n \text{ ($n$ even)}, E_7, E_8, F_4, G_2$. Indeed, for such a group $G$, $w_0=-1\in W$ and for any $w\in W$, $\dim V_{\alpha}=\dim V_{w\cdot\alpha}$ \cite[Planches I--IX]{bourbaki2007groupes}, \cite[Exercise and Table 1 page~59]{humphreys1990reflection}.
 \end{enumerate}
\end{example}

\begin{example}
 The two-dimensional representation of $\BoC^*$ given by $t\cdot (u,v)=(tu,t^{-2}v)$ is weakly symmetric but not symmetric.
\end{example}

For $\lambda\in X_*(T)$, we define $d_{\lambda}\coloneqq\#\CW^{\lambda=0}(V)-\#\CW^{\lambda=0}(\Fg)=\dim V^{\lambda}/G^{\lambda}$. The quantity $r_{\lambda}$ is defined in Lemma~\ref{lemma:smoothnessqproperp}. For later use, we give the following easy lemma.

\begin{lemma}
\label{lemma:numberweights}
 Let $V$ be a weakly symmetric representation of $G$. Then,
 \[
  d_{\lambda}+2r_{\lambda}=d\,.
 \]
\end{lemma}
\begin{proof}
 We have
 \[
  \CW(V)=\CW^{\lambda>0}(V)\sqcup\CW^{\lambda=0}(V)\sqcup\CW^{\lambda<0}(V)
 \]
 and by weak symmetry (Definition~\ref{definition:symmetricrepresentation}), $\#\CW^{\lambda>0}(V)=\#\CW^{\lambda<0}(V)$. Therefore, $\#\CW(V)=2\#\CW^{\lambda>0}(V)+\#\CW^{\lambda=0}(V)$. These equalities are also valid for $V=\Fg$ under the adjoint action of $G$ (which is a symmetric representation, Example~\ref{example:symmetricrepresentations}). We obtain the lemma by combining them together.
\end{proof}

\subsection{Parabolic induction}
\label{subsection:parabolicinduction}
For $\lambda\in X_*(T)$, we let $\BoQ_{V^{\lambda}/G^{\lambda}}^{\vir}\coloneqq \BoQ_{V^{\lambda}/G^{\lambda}}\otimes\SL^{-d_{\lambda}/2}=\IC(V^{\lambda}/G^{\lambda})$ be the intersection complex monodromic mixed Hodge module of the smooth stack $V^{\lambda}/G^{\lambda}$. We refer to \cite{davison2020cohomological} for the necessary background regarding monodromic mixed Hodge modules, and in particular regarding the square-root $\SL^{1/2}$ of the Tate twist $\SL$. The Tate twist $\SL\coloneqq\HO^*_{\rmc}(\mathbf{A}^1,\BoQ)$ is a pure weight $0$ complex of mixed Hodge structures, concentrated in degree $2$.  Alternatively, for the purposes of this paper, the reader may prefer to consider instead constructible sheaves. The Tate twist $-\otimes\SL$ is then replaced by the shift $[-2]$.

By smoothness of $q_{\lambda}$ (Lemma~\ref{lemma:smoothnessqproperp}), we have $q_{\lambda}^*\cong q_{\lambda}^!\otimes\SL^{r_{\lambda}}$ and thus we obtain
\[
 q_{\lambda}^*\BD\BoQ_{V^{\lambda}/G^{\lambda}}^{\vir}\cong\BD\BoQ_{V^{\lambda\geq 0}/P_{\lambda}}\otimes\SL^{r_{\lambda}+d_{\lambda}/2}
\]
which by adjunction provides us with the map
\begin{equation}
\label{equation:smoothpbq}
 \BD\BoQ_{V^{\lambda}/G^{\lambda}}^{\vir}\rightarrow(q_{\lambda})_*\BD\BoQ_{V^{\lambda\geq 0}/P_{\lambda}}\otimes\SL^{r_{\lambda}+d_{\lambda}/2}\,.
\end{equation}
Moreover, the map $p_{\lambda}$ is proper and representable (Lemma~\ref{lemma:smoothnessqproperp}), and so by dualizing the adjunction map
\[
 \BoQ_{V/G}\rightarrow (p_{\lambda})_*\BoQ_{V^{\lambda\geq 0}/P_{\lambda}}\,,
\]
and using $\BD (p_{\lambda})_*=(p_{\lambda})_!\BD$ and $(p_{\lambda})_*=(p_{\lambda})_!$, we obtain
\begin{equation}
\label{equation:properpfp}
 (p_{\lambda})_*\BD\BoQ_{V^{\lambda\geq 0}/P_{\lambda}}\rightarrow\BD\BoQ_{V/G}\,.
\end{equation}
We eventually obtain the map
\begin{equation}
\label{equation:sheafifiedinduction}
 \Ind_{\lambda}\colon (\imath_{\lambda})_*(\pi_{\lambda})_*\BD\BoQ_{V^{\lambda}/G^{\lambda}}^{\vir}\rightarrow \pi_*\BD\BoQ_{V/G}\otimes\SL^{r_{\lambda}+d_{\lambda}/2}
\end{equation}
by composing $(\imath_{\lambda})_*(\pi_{\lambda})_*$ applied to \eqref{equation:smoothpbq} with $\pi_*$ applied to \eqref{equation:properpfp} twisted by $\SL^{r_{\lambda}+d_{\lambda}/2}$. We used the commutativity of \eqref{equation:comminductiondiagram} to identify $\pi_*(p_{\lambda})_*\BD\BoQ_{V^{\lambda\geq 0}/P_{\lambda}}$ and $(\imath_{\lambda})_*(\pi_{\lambda})_*(q_{\lambda})_*\BD\BoQ_{V^{\lambda\geq 0}/P_{\lambda}}$.

\begin{lemma}
\label{lemma:preservescohdegrees}
 Let $V$ be a weakly symmetric representation of $G$. Then, the induction morphism $\Ind_{\lambda}\colon \HO^{*+d_{\lambda}}(V^{\lambda}/G^{\lambda})\rightarrow \HO^{*+d}(V/G)$ preserves shifted cohomological degrees, that is $\Ind_{\lambda}(\HO^{i+d_{\lambda}}(V^{\lambda}/G^{\lambda}))\subset \HO^{i+d}(V/G)$ for any $i\in\BoZ$.
\end{lemma}
\begin{proof}
 The induction at the level of cohomology vector spaces is given by applying the derived global section functor to the morphism of complexes \eqref{equation:sheafifiedinduction} (the sheafified induction). Moreover, by Lemma~\ref{lemma:numberweights}, $d/2=r_{\lambda}+d_{\lambda}/2$. Therefore, in the weakly symmetric case, the right-hand-side of \eqref{equation:sheafifiedinduction} is precisely $\pi_*\BD\BoQ_{V/G}^{\vir}$. Moreover, $\BD\BoQ_{V/G}^{\vir}\cong\BoQ_{V/G}^{\vir}$ and $\BD\BoQ_{V^{\lambda}/G^{\lambda}}^{\vir}\cong\BoQ_{V^{\lambda}/G^{\lambda}}^{\vir}$ since $V/G$ and $V^{\lambda}/G^{\lambda}$ are smooth and the chosen shifts/Tate twists (by the dimension) make the considered constant sheaves Verdier self-dual.
\end{proof}

\subsection{Augmentation}
\label{subsection:augmentation}
Let $\lambda\in X_*(T)$. We defined a representation $V^{\lambda}$ of $G^{\lambda}$ in \S\ref{subsection:inductiondiagram}. The induction morphisms, which we have defined for $(V,G)$, can also be defined in the exact same way for $(V^{\lambda},G^{\lambda})$. We can see $T$ as a maximal torus of $G^{\lambda}$. If $\mu\in X_*(T)$, we have $(V^{\lambda})^{\mu}=V^{\lambda}\cap V^{\mu}$, $(G^{\lambda})^{\mu}=G^{\lambda}\cap G^{\mu}$, $(V^{\lambda})^{\mu\geq 0}=V^{\lambda}\cap V^{\mu\geq 0}$ and $P_{\mu,\lambda}\coloneqq P_{\mu}\cap G^{\lambda}$. 

In this case, the induction diagram \eqref{equation:comminductiondiagram} becomes
\[\begin{tikzcd}
	& {(V^{\lambda})^{\mu\geq 0}/P_{\mu,\lambda}} \\
	{(V^{\lambda})^{\mu}/(G^{\lambda})^{\mu}} && {V^{\lambda}/G^{\lambda}} \\
	{(V^{\lambda})^{\mu}\cms (G^{\lambda})^{\mu}} && {V^{\lambda}\cms G^{\lambda}}
	\arrow["{q_{\lambda,\mu}}"', from=1-2, to=2-1]
	\arrow["{p_{\lambda,\mu}}", from=1-2, to=2-3]
	\arrow["{\pi_{\lambda,\mu}}"', from=2-1, to=3-1]
	\arrow["\pi_{\lambda}", from=2-3, to=3-3]
	\arrow["\imath_{\lambda,\mu}",from=3-1, to=3-3]
\end{tikzcd}\]

The definition of the induction map associated with this datum following \S\ref{subsection:parabolicinduction} is
\[
 \Ind_{\mu,\lambda}\colon \HO^*((V^{\lambda})^{\mu}/(G^{\lambda})^{\mu})\rightarrow \HO^*(V^{\lambda}/G^{\lambda})\,.
\]
 If $\mu\preceq\lambda$ \eqref{equation:orderintro}, then $(V^{\lambda})^{\mu}=V^{\mu}$ and $(G^{\lambda})^{\mu}=G^{\mu}$, so that the induction $\Ind_{\mu,\lambda}$ is
 \[
  \Ind_{\mu,\lambda}\colon \HO^*(V^{\mu}/G^{\mu})\rightarrow\HO^*(V^{\lambda}/G^{\lambda})\,.
 \]
 
\begin{lemma}
\label{lemma:gcdcocharacters}
  For any $\lambda,\mu\in X_*(T)$, one can find $\nu\in X_*(T)$ such that $V^{\nu}=V^{\mu}\cap V^{\lambda}$, $G^{\nu}=G^{\mu}\cap G^{\lambda}$, $(V^{\lambda})^{\nu\geq0}=(V^{\lambda})^{\mu\geq0}$ and $P_{\mu,\lambda}=P_{\nu,\lambda}$.
\end{lemma}
\begin{proof}
 The cocharacters $\lambda,\mu$ define a morphism
 \[
  \lambda\times\mu\colon \BoG_{\rmm}^{\times2}\rightarrow T\,.
 \]
We obtain $\nu$ by composing a general one-parameter subgroup $\zeta\colon\BoG_{\rmm}\rightarrow\BoG_{\rmm}^{\times2}$, $t\mapsto (t^a,t^b)$ with $\lambda\times\mu$. This guarantees the first two equalities of the Lemma. To guarantee the last two, it suffices to take a general one-parameter subgroup with the constraint $b>0$ by replacing $\zeta$ by $\zeta^{-1}$ if necessary.
\end{proof}

\subsection{Associativity}
\begin{proposition}
\label{proposition:associativity}
 Let $\lambda,\mu,\nu\in X_*(T)$ be such that $V^{\lambda}\subset V^{\mu}\subset V^{\nu}$, $G^{\lambda}\subset G^{\mu}\subset G^{\nu}$ and $(V^{\nu})^{\mu\geq 0}\subset(V^{\nu})^{\lambda\geq 0}$, $P_{\mu,\nu}\subset P_{\lambda,\nu}$. Then,
 \[
  \Ind_{\lambda,\nu}=\Ind_{\mu,\nu}\circ\Ind_{\lambda,\mu}\,.
 \]
\end{proposition}
\begin{proof}
The equality classically follows by base-change in the commutative diagram
\begin{equation}
\label{equation:associtivitydiagram}
\begin{tikzcd}
	&& {(V^{\nu})^{\lambda\geq0}/P_{\lambda,\nu}} \\
	& {(V^{\mu})^{\lambda\geq 0}/P_{\lambda,\mu}} && {(V^{\nu})^{\mu\geq 0}/P_{\mu,\nu}} \\
	{V^{\lambda}/G^{\lambda}} && {V^{\mu}/G^{\mu}} && {V^{\nu}/G^{\nu}}
	\arrow[from=1-3, to=2-2]
	\arrow[from=1-3, to=2-4]
	\arrow["{q_{\nu,\lambda}}"',bend right =30, from=1-3, to=3-1]
	\arrow["\lrcorner"{anchor=center, pos=0.125, rotate=-45}, draw=none, from=1-3, to=3-3]
	\arrow["{p_{\nu,\lambda}}", bend left=30, from=1-3, to=3-5]
	\arrow["{q_{\mu,\lambda}}"', from=2-2, to=3-1]
	\arrow["{p_{\mu,\lambda}}", from=2-2, to=3-3]
	\arrow["{q_{\nu,\mu}}"', from=2-4, to=3-3]
	\arrow["{p_{\nu,\mu}}", from=2-4, to=3-5]\,.
\end{tikzcd}
\end{equation}
\end{proof}

\subsection{Explicit formula for the parabolic induction}
\label{subsection:explicitformula}
If $\Ft_{\BoQ}$ is a rational form of the Lie algebra of $T$, we have $\HO^*_T(\pt)\cong\Sym(\Ft^*)$. Any character $\alpha\in X_*(T)$ may be identified with an element of $\Ft^*$ and hence $\HO^*_T(\pt)$ via $\alpha\mapsto \mathrm{d}\alpha(1)$ (see \S\ref{subsection:notations}). By abuse, we will still denote by $\alpha$ the corresponding element of $\HO^*_T(\pt)$.

For $\lambda\in X_*(T)$, we define the induction kernel
\[
 k_{\lambda}\coloneqq \frac{\prod_{\alpha\in\CW^{\lambda<0}(V)}\alpha}{\prod_{\alpha\in\CW^{\lambda<0}(\Fg)}\alpha}\in\Frac(\HO^*(\pt/T))\,.
\]

We let $W^{\lambda}\coloneqq N_{G^{\lambda}}(T)/T$ be the Weyl group of $G^{\lambda}$.

\begin{lemma}
\label{lemma:invariance}
 For any $\lambda\in X_*(T)$, $k_{\lambda}$ is $W^{\lambda}$-invariant.
\end{lemma}
\begin{proof}
 It suffices to prove that $W^{\lambda}$ sends $\CW^{\lambda<0}(V)\subset\CW(V)$ to itself, and similarly for $\CW^{\lambda<0}(\Fg)$. This comes from the invariance of the pairing between characters and cocharacters: for any $w\in W$, $\langle w\cdot\lambda,w\cdot\alpha\rangle=\langle\lambda,\alpha\rangle$ and if $w\in W^{\lambda}$, $w\cdot\lambda=\lambda$. Therefore, for any $\alpha\in X^*(T)$, $\langle\lambda,\alpha\rangle<0$ if and only if $\langle\lambda,w\cdot\alpha\rangle<0$.
\end{proof}

\begin{proposition}
\label{proposition:explicitformulainduction}
Let $\lambda\colon\BoG_{\rmm}\rightarrow T$ be a cocharacter and $f\in\HO^*(V^{\lambda}/G^{\lambda})\cong\HO^*(\pt/T)^{W^{\lambda}}$. We have
\[
 \Ind_{\lambda}(f)=\sum_{w\in W/W^{\lambda}}w\cdot(fk_{\lambda})=\frac{1}{\# W^{\lambda}}\sum_{w\in W}w\cdot(fk_{\lambda})\,.
\]
\end{proposition}
\begin{proof}
 The second equality follows from the $W^{\lambda}$-invariance of both $f$ (since $f\in \HO^*(V^{\lambda}/G^{\lambda})=\HO^*(V^{\lambda}/T)^{W^{\lambda}}$) and $k_{\lambda}$ (Lemma~\ref{lemma:invariance}). The first equality is a computation of Euler classes exactly as in the proof of \cite[Theorem 2]{kontsevich2011cohomological} in the case of quiver representations.
\end{proof}

More generally, let $\lambda,\mu\in X_*(T)$. We let
\[
 k_{\lambda,\mu}\coloneqq\frac{\prod_{\alpha\in\CW^{\lambda<0}(V^{\mu})}\alpha}{\prod_{\alpha\in\CW^{\lambda<0}(\Fg^{\mu})}\alpha}\in\Frac(\HO^*_T(\pt))\,.
\]
\begin{proposition}[Augmentation of Proposition~\ref{proposition:explicitformulainduction}]
 For $f\in\HO^*((V^{\mu})^{\lambda}/(G^{\mu})^{\lambda})\cong\HO^*(\pt/T)^{W^{\lambda}\cap W^{\mu}}$, we have
 \[
  \Ind_{\lambda,\mu}(f)=\sum_{w\in W^{\mu}/(W^{\lambda}\cap W^{\mu})}w(fk_{\lambda,\mu})\,.
 \]
\end{proposition}
\begin{proof}
 This follows immediately from Proposition~\ref{proposition:explicitformulainduction} applied to $V^{\mu}$ as a representation of $G^{\mu}$ and the cocharacter $\lambda$ of $T$, seen as a maximal torus of $G^{\mu}$, and the identification of the Weyl group of $(G^{\mu})^{\lambda}=G^{\lambda}\cap G^{\mu}$ with $W^{\lambda}\cap W^{\mu}$.
\end{proof}

\subsection{Parabolic induction for weakly symmetric representations}
\label{subsection:parindsymreps}

\begin{proposition}
 Let $V$ be a (weakly) symmetric representation of a reductive group $G$ and $\lambda\in X_*(T)$ a cocharacter of the maximal torus of $T$. Then, the representation $V^{\lambda}$ of $G^{\lambda}$ is also (weakly) symmetric.
\end{proposition}
\begin{proof}
 It is immediate to check that $\CW'(V^{\lambda})=\{\alpha\in\CW'(V)\mid\langle\lambda,\alpha\rangle=0\}$ and for any $\alpha\in\CW'(V^{\lambda})$, we can identify the weight spaces $V^{\lambda}_{\alpha}=V_{\alpha}$. Therefore, the (weak) symmetry of $V$ implies that of $V^{\lambda}$.
\end{proof}

Recall the order relation on $X_*(T)$ defined in the introduction \eqref{equation:orderintro}
\[
 \lambda\preceq \mu\overset{\text{def.}}{\iff} 
 \left\{
 \begin{aligned}V^{\lambda}\subset V^{\mu}\\
 \Fg^{\lambda}\subset\Fg^{\mu}\,.
 \end{aligned}
 \right.
\]
for any $\lambda,\mu\in X_*(T)$.

We obtain an equivalence relation on the set $X_*(T)$
\[
 \lambda\sim \mu\overset{\text{def.}}{\iff} \lambda\preceq\mu \text{ and }\mu\preceq\lambda\,.
\]
We let $\SP_V\coloneqq X_*(T)/\sim$ be the quotient of $X_*(T)$ by this equivalence relation. If $\lambda\in X_*(T)$, $\overline{\lambda}$ denotes its class in $\SP_V$. It has an induced order relation still denoted by $\preceq$:
\[
 \overline{\lambda}\preceq\overline{\mu}\iff \lambda\preceq\mu\,.
\]
It is immediate that this is well-defined, i.e. the definition does not depend on the representatives $\lambda,\mu$ of $\overline{\lambda},\overline{\mu}$.

\begin{lemma}
\label{lemma:numberpositiverootsrep}
 Let $V$ be a weakly symmetric representation of a reductive group $G$. If $\lambda,\mu\in X_*(T)$ are such that $\lambda\sim\mu$, then 
 \begin{enumerate}
  \item $\#\CW^{\lambda>0}(V)=\#\CW^{\mu>0}(V)=\#\CW^{\lambda<0}(V)=\#\CW^{\mu<0}(V)$,
  \item $\#\CW^{\lambda>0}(\Fg)=\#\CW^{\mu>0}(\Fg)=\#\CW^{\lambda<0}(\Fg)=\#\CW^{\mu<0}(\Fg)$.
 \end{enumerate}
\end{lemma}
\begin{proof}
 We prove the statement for $V$. The statement for $\Fg$ is obtained by specialising $V=\Fg$ since $\Fg$ under the adjoint action of $G$ is a symmetric representation of $G$ (by the existence of invariant pairings, nondegenerate up to the center of $\Fg$, Example~\ref{example:symmetricrepresentations}).
 
 Since $V$ is weakly symmetric, we have $\#\CW^{\lambda>0}(V)=\#\CW^{\lambda<0}(V)$ and moreover, $\#\CW(V)=\#\CW^{\lambda>0}(V)+\#\CW^{\lambda<0}(V)+\CW(V^{\lambda})$. Since we have the same equalities with $\lambda$ replaced by $\mu$ and $\CW^{\lambda=0}(V)=\CW^{\mu=0}(V)$, we obtain $\#\CW^{\lambda>0}(V)=\#\CW^{\mu>0}(V)$. The other equalities follow.
\end{proof}

\begin{lemma}
\label{lemma:signkernelsim}
 If $\lambda,\mu\in X_*(T)$ are such that $\lambda\sim\mu$, then $k_{\lambda}=\xi k_{\mu}$ for some $\xi\in\BoQ ^*$.
\end{lemma}
\begin{proof}
 If $\lambda\sim\mu$, we have $\CW^{\lambda>0}(V)\sqcup\CW^{\lambda<0}(V)=\CW^{\mu>0}(V)\sqcup\CW^{\mu<0}(V)$ and $\CW^{\lambda>0}(\Fg)\sqcup\CW^{\lambda<0}(\Fg)=\CW^{\mu>0}(\Fg)\sqcup\CW^{\mu<0}(\Fg)$ and since $V$ and $\Fg$ are weakly symmetric representations of $G$, $\overline{\CW^{\lambda>0}(V)}=-\overline{\CW^{\lambda<0}(V)}$, $\overline{\CW^{\mu>0}(V)}=-\overline{\CW^{\mu<0}(V)}$ and similarly for $\Fg$. Therefore, for any $\alpha\in\CW^{\lambda<0}(V)$ (resp. $\CW^{\lambda<0}(\Fg)$), there exists $\xi_{\alpha}\in\BoQ^*$ (resp. $\zeta_{\alpha}\in\BoQ^*$) such that
 \[
  \begin{matrix}
   \CW^{\lambda<0}(V)&\rightarrow&\CW^{\mu<0}(V)\\
   \alpha&\mapsto &\xi_{\alpha}\alpha
  \end{matrix}
 \]
and
 \[
  \begin{matrix}
   \CW^{\lambda<0}(\Fg)&\rightarrow&\CW^{\mu<0}(\Fg)\\
   \alpha&\mapsto &\zeta_{\alpha}\alpha
  \end{matrix}
 \]
 are bijections.
 Then,
 \[
 \frac{\prod_{\alpha\in\CW^{\lambda<0}(V)}\xi_{\alpha}}{\prod_{\alpha\in\CW^{\lambda<0}(\Fg)}\zeta_{\alpha}} k_{\lambda}=k_{\mu}\,.
 \]
We may therefore set $\xi\coloneqq \frac{\prod_{\alpha\in\CW^{\lambda<0}(\Fg)}\zeta_{\alpha}}{\prod_{\alpha\in\CW^{\lambda<0}(V)}\xi_{\alpha}}$.
\end{proof}

\begin{lemma}
\label{lemma:sameimageoverline}
 We assume that $V$ is a weakly symmetric representation of $G$. Then, for any $\lambda, \mu\in X_*(T)$ such that $\overline{\lambda}=\overline{\mu}$, the induction maps $\Ind_{\lambda}$ and $\Ind_{\mu}$ differ by a non-zero scalar. In particular, the images of $\Ind_{\lambda}$ and $\Ind_{\mu}$ inside $\HO^{*+d}(V/G)$ coincide.
\end{lemma}
\begin{proof}
 We use the explicit formula for the induction (Proposition~\ref{proposition:explicitformulainduction}). Let $f\in \HO^*(V^{\lambda}/G^{\lambda})=\HO^*(V^{\mu}/G^{\mu})$. We have
 \[
  \Ind_{\lambda}(f)=\sum_{w\in W/W^{\lambda}}w\cdot(fk_{\lambda}),\quad  \Ind_{\mu}(f)=\sum_{w\in W/W^{\mu}}w\cdot(fk_{\mu})\,.
 \]
Moreover, $G^{\lambda}=G^{\mu}$, $W^{\lambda}=W^{\mu}$ and by Lemma~\ref{lemma:signkernelsim}, $k_{\lambda}=\xi k_{\mu}$. This concludes the proof.
\end{proof}

The $W$-action on $X_*(T)$ induces a $W$-action on $\SP_V$. For $\lambda\in X_*(T)$, we let $\tilde{\lambda}\in \SP_V/W$ be its $W$-orbit.

\begin{lemma}
\label{lemma:sameimageWeylgroup}
 Let $\lambda,\mu\in X_*(T)$ be such that $\tilde{\lambda}=\tilde{\mu}$. Then, the images of $\Ind_{\lambda}$ and $\Ind_{\mu}$ coincide.
\end{lemma}
\begin{proof}
 By Lemma~\ref{lemma:sameimageoverline}, we may assume that for some $w\in W$, $\mu=w\cdot\lambda$. Then, $\CW^{\mu<0}(V)=w\cdot\CW^{\lambda<0}(V)$ and $\CW^{\mu<0}(\Fg)=w\cdot\CW^{\lambda<0}(\Fg)$. Therefore, $w(k_{\lambda})=k_{\mu}$. Using the explicit formula for the induction product (Proposition~\ref{proposition:explicitformulainduction}), we obtain
 \[
  \Ind_{\mu}(f)=\Ind_{\lambda}(w^{-1}(f))\,,
 \]
and therefore, since $w^{-1}$ induces an isomorphism $\HO^*(V^{\mu}/G^{\mu})=\HO^*(V^{w\cdot\lambda}/G^{w\cdot\lambda})\cong\HO^*(V^{\lambda}/G^{\lambda})$, the images of $\Ind_{\lambda}$ and $\Ind_{\mu}$ coincide.
\end{proof}

\section{Action of tautological classes on the cohomology of the stack}
Let $G$ be a reductive group and $V$ a representation of $G$. We let $\pi\colon V/G\rightarrow V\cms G$ be the good moduli space map. We let $\kappa\colon V\cms G\rightarrow\pt$ be the projection to the point.

Then, for any line bundle $\CL$ over $V/G$, we have an operation $\pi_*\BoQ_{V/G}\rightarrow\pi_*\BoQ_{V/G}[-2]$ of multiplication by the first Chern class $c_1(\CL)$. More precisely, we have the following lemma.

\begin{lemma}
\label{lemma:actiontautologicalclasses}
There is a morphism of graded rings
 \[
  \HO^*(\pt/G)\xrightarrow{a} \RHom(\pi_*\BoQ_{V/G},\pi_*\BoQ_{V/G})
 \]
such that the composition $\kappa_*\circ a$ where $\kappa_*\colon\RHom(\pi_*\BoQ_{V/G},\pi_*\BoQ_{V/G})\rightarrow\RHom(\kappa_*\pi_*\BoQ_{V/G},\kappa_*\pi_*\BoQ_{V/G})\cong\RHom(\HO^*(\pt/G),\HO^*(\pt/G))$, send $a\in\HO^*(\pt/G)$ to the endomorphism of $\HO^*(\pt/G)$ given by the cup product with $a$.
\end{lemma}
\begin{proof}
 We have a canonical adjunction isomorphism
 \[
  \RHom(\pi_*\BoQ_{V/G},\pi_*\BoQ_{V/G})\cong\RHom(\pi^*\pi_*\BoQ_{V/G},\BoQ_{V/G})\,.
 \]
 By precomposing with the counit map $\pi^*\pi_*\BoQ_{V/G}\rightarrow\BoQ_{V/G}$, we obtain the morphism
 \[
  \RHom(\BoQ_{V/G},\BoQ_{V/G})\cong\HO^*(V/G)\cong\HO^*(\pt/G)\rightarrow \RHom(\pi_*\BoQ_{V/G}\,,\pi_*\BoQ_{V/G}).
 \]
 
 The last statement of the theorem follows from the functoriality of adjunctions.
\end{proof}

We now state and prove an elementary lemma regarding Cartan subalgebras of reductive Lie algebras.
\begin{lemma}
\label{lemma:inducedtorus}
 Let $\Fg$ be a reductive Lie algebra, $\Fh\subset\Fg$ a Lie ideal and $\Ft$ a Cartan subalgebra. Then, $\Fh\cap\Ft$ is a Cartan subalgebra of $\Fh$.
\end{lemma}
\begin{proof}
 Let $\pi\colon\Fg\rightarrow\Fg/\Fh$ be the projection. Since $\Fg$ is reductive, we have a (noncanonical) decomposition $\Fg\cong\Fh\oplus(\Fg/\Fh)$ as Lie algebras. If $\Ft'$ is a Cartan subalgebra of $\Fh$ and $\Ft''$ is a Cartan subalgebra of $\Fg/\Fh$, then $\Ft'\oplus\Ft''$ is a Cartan subalgebra of $\Fg$. Moreover, $\pi(\Ft)\cong \Ft/\Ft\cap\Fh$ is an Abelian subalgebra of $\Fg/\Fh$ and $\Ft\cap\Fh$ is an Abelian subalgebra of $\Fh$. By maximality of $\dim\Ft'$, we have
 \[
  \dim \pi(\Ft)=\dim \Ft-\dim(\Ft\cap\Fh)\geq \dim \Ft-\dim \Ft'=\dim\Ft''\,.
 \]
 By maximality of $\dim\Ft''$, we have equality in this chain of inequalities. This means that $\dim(\Ft\cap\Fh)=\dim\Ft'$ and therefore, $\Ft\cap\Fh$ is a Cartan subalgebra of $\Fh$.
\end{proof}

We have the group version of Lemma~\ref{lemma:inducedtorus}.
\begin{lemma}
 Let $G$ be a reductive group, $H\subset G$ a normal subgroup and $T\subset G$ a maximal torus. Then, the neutral component of $T\cap H$ is a maximal torus of $H$.
\end{lemma}
\begin{proof}
 This follows from Lemma~\ref{lemma:inducedtorus} by taking the Lie algebras of all groups appearing.
\end{proof}

\begin{example}
In the group situation, it is necessary to consider the neutral component of $T\cap H$ as shown by the elementary case of $H=\{\pm 1\}\subset\rmSL_2(\BoC)$, and where $T\subset\rmSL_2(\BoC)$ is the standard torus of diagonal matrices.
\end{example}

\begin{proposition}
\label{proposition:actions}
 Let $G$ be a reductive group. For any normal subgroup $H\subset G$ acting trivially on $V$, we have (non-canonical) actions of $\HO^*(\pt/H)$ on $\HO^*(\pt/G)\cong \HO^*(V/G)$ and $\pi_*\BoQ_{V/G}$.
\end{proposition}
\begin{proof}
 By choosing a decomposition $\Lie(G)\cong \Lie(H)\oplus\Lie(G/H)$, we obtain a decomposition $\HO^*(\pt/G)\cong\HO^*(\pt/(G/H))\otimes\HO^*(\pt/H)$, giving an action of $\HO^*(\pt/H)$ on $\HO^*(\pt/G)$ by cup-product on the second summand. At the sheaf level, the action of the lemma is obtained by pre-composing the morphism $a$ of Lemma~\ref{lemma:actiontautologicalclasses} with the inclusion $\HO^*(\pt/H)\rightarrow\HO^*(\pt/G)$.
\end{proof}

\begin{lemma}
\label{lemma:naturalinclusion}
 Let $G$ be an algebraic group and $H\subset G$ a normal subgroup. Then, we have a natural inclusion $\HO^*(\pt/(G/H))\rightarrow\HO^*(\pt/G)$.
\end{lemma}
\begin{proof}
 The map $\HO^*(\pt/(G/H))\rightarrow\HO^*(\pt/G)$ is the pullback in cohomology for the morphism of stacks $\pt/G\rightarrow\pt/(G/H)$ induced by the surjective morphism of groups $G\rightarrow G/H$.
\end{proof}

\begin{proposition}
\label{proposition:decompositioncohomology}
 Let $G$ be a reductive algebraic group, and $V$ a finite-dimensional representation of $G$. Let $H$ be a normal subgroup of $G$ acting trivially on $V$. Then, the morphism
 \[
  \HO^*(\pt/H)\otimes\HO^*(V/(G/H))\rightarrow\HO^*(V/G)
 \]
 obtained by combining the natural pullback map $\HO^*(V/(G/H))\rightarrow\HO^*(V/G)$ coming from the morphism of stacks $V/G\rightarrow V/(G/H)$ (Lemma~\ref{lemma:naturalinclusion}) with the $\HO^*(\pt/H)$-action on $\HO^*(V/G)$ (Proposition~\ref{proposition:actions}) is an isomorphism.
\end{proposition}
\begin{proof}
We let $\Fg=\Lie(G)$ and $\Fh=\Lie(H)$. We let $\Ft\subset \Fg$ be a Cartan subalgebra. We have $\Fg\cong\Fh\oplus(\Fg/\Fh)$ and therefore a decomposition of the Weyl group $W_{\Fg}\cong W_{\Fh}\times W_{\Fg/\Fh}$. We let $\pi\colon \Fg\rightarrow\Fg/\Fh$ be the projection. The morphism of the proposition can then be identified with the isomorphism
\[
 \Sym((\Ft\cap\Fh)^*)^{W_{\Fh}}\otimes\Sym(\pi(\Ft)^*)^{W_{\Fg/\Fh}}\rightarrow\Sym(\Ft^*)^{W}\,,
\]
see also Lemma~\ref{lemma:inducedtorus}.
\end{proof}

\section{Cohomological integrality for weakly symmetric representations of reductive groups}

We fix a weakly symmetric representation $V$ (Definition~\ref{definition:symmetricrepresentation}) of a connected reductive group $G$ and a maximal torus $T\subset G$.

\subsection{Weyl groups}
\label{subsection:Weylgroups}
Let $\lambda\colon \BoG_{\rmm}\rightarrow T$. We let $G^{\lambda}\subset G$ be the corresponding Levi subgroup (\S\ref{subsection:inductiondiagram}) and $G_{\lambda}\coloneqq Z(G^{\lambda})\cap\ker(G^{\lambda}\rightarrow\GL(V^{\lambda}))$ the intersection of the kernel of the action of $G^{\lambda}$ on $V^{\lambda}$ and the center of $G^{\lambda}$. We let $W^{\lambda}=N_{G^{\lambda}}(T)$ be the Weyl group of $G^{\lambda}$ and $W_{\lambda}\coloneqq \{w\in W\mid V^{w\cdot\lambda}=V^{\lambda}\text{ and }G^{w\cdot\lambda}=G^{\lambda}\}=\{w\in W\mid w\cdot\overline{\lambda}=\overline{\lambda}\}$. The group $W_{\lambda}$ may be thought of as a kind of relative Weyl group.

\begin{lemma}
\label{lemma:G_lambdacohomology}
 Let $G_{\lambda}^{\circ}$ be the neutral component of $G_{\lambda}$. Let $k=\dim G_{\lambda}$. Then, $G_{\lambda}^{\circ}\cong(\BoC^*)^k$ for some $k\geq 0$ and the natural pullback map $\HO^*(\pt/G_{\lambda})\rightarrow\HO^*(\pt/G_{\lambda}^{\circ})$ is an isomorphism.
\end{lemma}
\begin{proof}
 The group $G_{\lambda}$ is a subgroup of the center of $G^{\lambda}$, and so it is a diagonalisable algebraic group. Therefore, its neutral component is a torus. We have a splitting $G\cong G_{\lambda}^{\circ}\times\pi_0(G_{\lambda})$ \cite[Theorem p.104]{humphreys2012linear} and so, by the K\"unneth formula,
 \[
  \HO^*(\pt/G_{\lambda})\cong\HO^*(\pt/G_{\lambda}^{\circ})\otimes\HO^*(\pt/\pi_0(G_{\lambda}))\,.
 \]
 Since $\pi_0(G_{\lambda})$ is a finite group, $\HO^*(\pt/\pi_0(G_{\lambda}))\cong\HO^*(\pt)^{\pi_0(G_{\lambda})}\cong\BoQ$, which concludes.
\end{proof}

\begin{lemma}
\label{lemma:Wlambdasubgroup}
 The group $W_{\lambda}$ is a subgroup of $W$ and it contains $W^{\lambda}$ as a normal subgroup.
\end{lemma}
\begin{proof}
 We let $\dot{w}\in N_G(T)$ be a lift of $w\in W$. Then, $V^{w\cdot\lambda}=\dot{w}V^{\lambda}$. The first statement of the lemma follows then immediately from the fact that if $w,w'\in W$, and $\dot{w}, \dot{w'}\in N_G(T)$ are respective lifts, then, $\dot{w}\dot{w'}$ is a lift of $ww'$ and $\dot{w}^{-1}$ a lift of $w^{-1}$.

 We prove that $W^{\lambda}\subset W_{\lambda}$. Let $w\in W^{\lambda}$ and $\dot{w}\in N_{G^{\lambda}}(T)$ a lift of $w$. Then, $\dot{w}G^{\lambda}\dot{w}^{-1}=G^{w\cdot\lambda}=G^{\lambda}$ and $\dot{w}V^{\lambda}=V^{w\cdot\lambda}=V^{\lambda}$ and so $w\in W_{\lambda}$. This proves the inclusion.

 We now show that $W^{\lambda}$ is a normal subgroup of $W_{\lambda}$. It suffices to check that if $w\in W_{\lambda}$ and $w'\in W^{\lambda}$, and $\dot{w}\in N_G(T)$, $\dot{w'}\in N_{G^{\lambda}}(T)$ are any lifts, then $\dot{w}\dot{w'}\dot{w}^{-1}\in G^{\lambda}$. We have $\dot{w}\dot{w'}\dot{w}^{-1}\in \dot{w}G^{\lambda}\dot{w}^{-1}=G^{w\cdot \lambda}=G^{\lambda}$. This concludes the proof.
\end{proof}

\begin{corollary}
\label{corollary:W_lambdaG^lambda}
 The group $W_{\lambda}$ acts on $\HO^*(\pt/T)^{W^{\lambda}}\cong \HO^*(\pt/G^{\lambda})$.
\end{corollary}
\begin{proof}
 The Weyl group $W$ of $G$ acts on $\HO^*(\pt/T)$. We just have to check that the induced action of the subgroup $W_{\lambda}$ preserves $W^{\lambda}$-invariants. This is a consequence of Lemma~\ref{lemma:Wlambdasubgroup}.
\end{proof}

\begin{lemma}
\label{lemma:W_lambdaG_lambda}
 The group $W_{\lambda}$ acts naturally on $\HO^*(\pt/G_{\lambda})$. The subgroup $W^{\lambda}\subset W_{\lambda}$ acts trivially on $\HO^*(\pt/G_{\lambda})$.
\end{lemma}
\begin{proof}
 By definition, $G_{\lambda}$ is preserved by $W_{\lambda}$, as $G_{\lambda}$ only depends on $\overline{\lambda}\in\SP_V$ and $W_{\lambda}$ is the stabilizer of $\overline{\lambda}$. The second statement follows from the fact that $G_{\lambda}$ is contained in the center of $G^{\lambda}$.
\end{proof}

We let $\overline{G^{\lambda}}\coloneqq G^{\lambda}/G_{\lambda}$. Since $G_{\lambda}$ is a normal subgroup of $G^{\lambda}$, $\overline{G^{\lambda}}$ is a reductive group. Since $G_{\lambda}$ is contained in the center of $G^{\lambda}$, the Weyl group of $\overline{G^{\lambda}}$ is isomorphic to $W^{\lambda}$.

\begin{lemma}
\label{lemma:W_lambdaoverlineGlambda}
 The group $W_{\lambda}$ acts naturally on $\HO^*(V^{\lambda}/\overline{G^{\lambda}})$. Moreover, the $W_{\lambda}$-action on $\HO^*(\pt/G^{\lambda})\cong\HO^*(V^{\lambda}/G^{\lambda})$ obtained in Corollary~\ref{corollary:W_lambdaG^lambda} and the $W_{\lambda}$-action on $\HO^*(V^{\lambda}/\overline{G^{\lambda}})\otimes \HO^*(\pt/G_{\lambda})$ obtained by the tensor product of the actions given in Lemmas~\ref{lemma:W_lambdaG_lambda} and~\ref{lemma:W_lambdaoverlineGlambda} coincide via the isomorphism given in Proposition~\ref{proposition:decompositioncohomology}.
\end{lemma}
\begin{proof}
 We have $\HO^*(V^{\lambda}/\overline{G^{\lambda}})\cong \HO^*(\pt/(T/G_{\lambda}^{\circ}))^{W^{\lambda}}$. The restriction to $W_{\lambda}$ of the $W$ action on $T$ preserves $G_{\lambda}$ since $G_{\lambda}$ is in the center of $G^{\lambda}$. Therefore, we obtain a $W_{\lambda}$-action on $T/G_{\lambda}^{\circ}$. It induces a $W_{\lambda}$-action on $\HO^*(\pt/(T/G_{\lambda}^{\circ}))$. The fact that $W_{\lambda}$ normalises $W^{\lambda}$ (Lemma~\ref{lemma:Wlambdasubgroup}) implies that $\HO^*(\pt/(T/G_{\lambda}^{\circ}))^{W^{\lambda}}$ is preserved by the $W_{\lambda}$-action.

 The tensor product decomposition of the action follows immediately from this and Lemma~\ref{lemma:W_lambdaG_lambda}.
\end{proof}

\begin{proposition}
\label{proposition:characterW_lambda}
 There exists a character $\varepsilon_{V,\lambda}\colon W_{\lambda}\rightarrow\BoQ^*$ such that for any $w\in W_{\lambda}$, $w(k_{\lambda})=\varepsilon_{V,\lambda}(w)^{-1}k_{\lambda}\in\HO^*_T(\pt)$.
\end{proposition}
\begin{proof}
 It suffices to prove the existence and unicity of the rational number $\varepsilon_{V,\lambda}(w)\in\BoQ^*$ satisfying the equality of the proposition, for any $w\in W_{\lambda}$. The multiplicativity then follows from the unicity of this sign.
 
 For $w\in W_{\lambda}$, the cocharacters $\lambda,w^{-1}\cdot\lambda$ of $T$ are such that $\overline{\lambda}=\overline{w^{-1}\cdot\lambda}$. Therefore, by Lemma~\ref{lemma:signkernelsim}, there exists $\xi\in\BoQ^*$ such that $w(k_{\lambda})=k_{w^{-1}\cdot\lambda}=\xi k_{\lambda}$. We set $\varepsilon_{V,\lambda}(w)\coloneqq \xi^{-1}$. This proves the existence. The unicity of $\xi$ is straightforward.
\end{proof}

\begin{proposition}
 Let $\lambda\in X_*(T)$, $w\in W_{\lambda}$ and $f\in\HO^*(V^{\lambda}/G^{\lambda})$. Then, $\Ind_{\lambda}(w\cdot f)=\varepsilon_{V,\lambda}(w)\Ind_{\lambda}(f)$.
\end{proposition}
\begin{proof}
 By Proposition~\ref{proposition:explicitformulainduction}, we have
 \[
  \begin{aligned}
  \Ind_{\lambda}(w\cdot f)&=\frac{1}{\# W^{\lambda}}\sum_{w'\in W}w'(wf)w'(k_{\lambda})\\
  &=\frac{1}{\# W^{\lambda}}\sum_{w''\in W}(w''f)w''(w^{-1}k_{\lambda})\\
  &=\frac{1}{\# W^{\lambda}}\sum_{w''\in W}(w''f)w''(\varepsilon_{V,\lambda}(w)k_{\lambda}) \quad\text{ by Proposition~\ref{proposition:characterW_lambda}}
   \end{aligned}
 \]
 which concludes.
\end{proof}

We now state an elementary result from the representation theory of finite groups for later use.

\begin{proposition}
\label{proposition:projector}
 Let $G$ be a finite group $\chi\colon G\rightarrow\BoC^*$ a character of $G$ and $V$ a finite-dimensional representation of $G$. Then, the formula
 \[
  p_{\chi}(v)=\frac{1}{\# G}\sum_{g\in G}\chi(g)^{-1}g\cdot v
\]
is the projector onto the $\chi$-isotypic component $V^{\chi}$ of $V$.
\end{proposition}
\begin{proof}
This is \cite[\S2.6, Th\'eor\`eme~8]{serre1971representation}.
\end{proof}

\subsection{Proof of the cohomological integrality theorem}
\label{subsection:proofcohint}
We let $V$ be a weakly symmetric representation of a reductive group $G$. In the rest of this section, we prove Theorem~\ref{theorem:cohintabsolute}.

For $\lambda\in X_*(T)$, we define the $\BoZ$-graded vector space
\[
 \CH_{\lambda}^*\coloneqq \HO^{*}(V^{\lambda}/G^{\lambda})\,.
\]
When $\lambda=0$ is the trivial cocharacter, we write $\CH\coloneqq\CH_0=\HO^{*}(V/G)$. More generally, for any pair $\lambda,\mu$, we define $\CH_{\lambda,\mu}\coloneqq \HO^{*}(V^{\lambda}\cap V^{\mu}/G^{\lambda}\cap G^{\mu})$. By \S\ref{subsection:augmentation}, we have induction maps $\Ind_{\lambda}\colon \CH_{\lambda}\rightarrow\CH$ for $\lambda\in X_*(T)$ and more generally, $\Ind_{\mu,\lambda}\colon \CH_{\lambda,\mu}\rightarrow\CH_{\lambda}$ for $\lambda,\mu\in X_*(T)$.

For $\lambda\in X_*(T)$, we choose a $W_{\lambda}$-invariant splitting $\Fg^{\lambda}=\Fg_{\lambda}\oplus(\Fg^{\lambda}/\Fg_{\lambda})$, where $\Fg_{\lambda}\coloneqq\Lie(G_{\lambda})$ is the Lie algebra of $G_{\lambda}$ (\S\ref{subsection:Weylgroups}). It induces a splitting of the maximal torus $\Ft=\Fg_{\lambda}\oplus(\Ft/\Fg_{\lambda})$ (Lemma~\ref{lemma:inducedtorus}). We let
\[\CH_{\lambda}^{\prim}\coloneqq \HO^*(V^{\lambda}/(G^{\lambda}/G_{\lambda}^{\circ}))\,,\] so that $\CH_{\lambda}\cong \CH_{\lambda}^{\prim}\otimes\HO^*(\pt/G_{\lambda})$ (Proposition~\ref{proposition:decompositioncohomology}). We let
\[\CA_{\lambda}\coloneqq \HO^*(V^{\lambda}/T) \quad\text{and}\quad \CA_{\lambda}^{\prim}\coloneqq \HO^*(V^{\lambda}/(T/G_{\lambda}^{\circ}))\] so that $\CA_{\lambda}\cong\CA_{\lambda}^{\prim}\otimes\HO^*(\pt/G_{\lambda})$, $\CH_{\lambda}=\CA_{\lambda}^{W^{\lambda}}$ and $\CH_{\lambda}^{\prim}\cong(\CA_{\lambda}^{\prim})^{W^{\lambda}}$.
  
 For each $\lambda\in X_*(T)$, we let $J_{\lambda}$ be the smallest $W^{\lambda}$-stable $\CA_{\lambda}^{\prim}$-submodule of the localisation \[\CA_{\lambda}^{\prim,\loc}\coloneqq\CA_{\lambda}^{\prim}[\prod_{\alpha\in\CW(\Fg^{\lambda})\setminus\{0\}}\alpha^{-1}]\] containing
 \[
  k_{\mu,\lambda}=\frac{\prod_{\alpha\in\CW^{\mu<0}(V^{\lambda})}\alpha}{\prod_{\alpha\in\CW^{\mu<0}(\Fg^{\lambda})}\alpha}
 \]
 for all $\mu\in X_*(T)$ such that $\lambda\not\preceq \mu$ for the order relation $\preceq$ defined in \S\ref{subsection:parindsymreps}.
 
 We let \[\CA_{\lambda}^{\loc}\coloneqq\CA_{\lambda}[\prod_{\alpha\in\CW(\Fg^{\lambda})\setminus\{0\}}\alpha^{-1}]\cong\CA_{\lambda}^{\prim,\loc}\otimes\HO^*(\pt/G_{\lambda})\] and
 \[\quad\CH_{\lambda}^{\loc}\coloneqq(\CA_{\lambda}^{\loc})^{W^{\lambda}}\cong\CH_{\lambda}^{\prim}[\prod_{\alpha\in\CW(\Fg^{\lambda})\setminus\{0\}}\alpha^{-1}]\otimes\HO^*(\pt/G_{\lambda})
 \]
 (for this last isomorphism, we use the fact that $W^{\lambda}$ acts trivially on $\HO^*(\pt/G_{\lambda})$ as $G_{\lambda}$ is in the center of $G^{\lambda}$, Lemma~\ref{lemma:W_lambdaG_lambda}).

 \begin{lemma}
 \label{lemma:characterisationJlambda}
 We can characterise $J_{\lambda}$ as the smallest $W^{\lambda}$-stable $\CA_{\lambda}^{\prim}$-submodule of $\CA^{\prim,\loc}$ containing
 \[
  k_{\mu,\lambda}=\frac{\prod_{\alpha\in\CW^{\mu<0}(V^{\lambda})}\alpha}{\prod_{\alpha\in\CW^{\mu<0}(\Fg^{\lambda})}\alpha}
 \]
 for all $\mu\in X_*(T)$ such that $\mu\prec\lambda$.
 \end{lemma}
 \begin{proof}
  If $\mu\prec\lambda$, then we have $\dim V^{\mu}\leq\dim V^{\lambda}$ and $\dim \Fg^{\mu}\leq\dim\Fg^{\lambda}$ with at least one strict inequality (where the order relation $\preceq$ is defined in \S\ref{subsection:parindsymreps}). If by contradiction we have $\lambda\preceq \mu$ then, we have
  \[
   \dim V^{\lambda}\leq \dim V^{\mu}\,, \quad \dim\Fg^{\lambda}\leq\dim \Fg^{\mu}\,.
  \]
Therefore, if $\mu\prec\lambda$, then $\lambda\not\preceq\mu$ and by definition, $k_{\mu,\lambda}\in J_{\lambda}$.
  
  Conversely, if $\lambda\not\preceq\mu$, we let $\nu\in X_*(T)$ be such that $V^{\nu}=V^{\lambda}\cap V^{\mu}$, $\Fg^{\nu}=\Fg^{\lambda}\cap\Fg^{\mu}$, $(V^{\lambda})^{\nu\geq 0}=(V^{\lambda})^{\mu\geq 0}$ and $(\Fg^{\lambda})^{\nu\geq 0}=(\Fg^{\lambda})^{\mu\geq 0}$ (Lemma~\ref{lemma:gcdcocharacters}). We have $V^{\nu}\subset V^{\lambda}$ and $\Fg^{\nu}\subset\Fg^{\lambda}$ with at least one strict inclusion. Therefore, $\nu\prec\lambda$ and moreover, $k_{\nu,\lambda}=k_{\mu,\lambda}$. This proves the lemma.
 \end{proof}

 \begin{lemma}
 \label{lemma:invariantsJH}
  We have $J_{\lambda}^{W^{\lambda}}\subset\CH_{\lambda}^{\prim}$. Moreover, $J_{\lambda}\cap\CH_{\lambda}^{\prim}=J_{\lambda}^{W^{\lambda}}$.
 \end{lemma}
\begin{proof}
 This comes from the fact that the induction product \S\ref{subsection:parabolicinduction} is well-defined on the cohomology, without requiring localisation, even though the explicit formula (Proposition~\ref{proposition:explicitformulainduction}) involves rational fractions. More precisely, $J_{\lambda}^{W^{\lambda}}$ is linearly generated by the elements
 \[
  \sum_{w\in W^{\lambda}}w\cdot(fk_{\mu,\lambda})
 \]
for $f\in \CA_{\lambda}^{\prim}$ and $\mu\prec\lambda$, by Lemma~\ref{lemma:characterisationJlambda}. These are averages over $W^{\lambda}$ of elements of $J_{\lambda}$.
We can rewrite this sum as
\[
 \sum_{\overline{w'}\in W^{\lambda}/W^{\mu}}w'\cdot\left[\left(\sum_{w\in W^{\mu}}(w\cdot f)\right)k_{\mu,\lambda}\right]
\]
since $k_{\mu,\lambda}$ is $W^{\mu}$-invariant (Lemma~\ref{lemma:invariance}). This is (Proposition~\ref{proposition:explicitformulainduction}) the formula for
\[
 I\coloneqq\Ind_{\mu,\lambda}\left(\sum_{w\in W^{\mu}}(w\cdot f)\right)\,.
\]
Since $\sum_{w\in W^{\mu}}(w\cdot f)$ is polynomial, its induction $I$ is also polynomial. Therefore, it belongs to $\CH_{\lambda}$. Moreover, $I\in J_{\lambda}\subset \CA_{\lambda}^{\prim}[\prod_{\alpha\in\CW(\Fg^{\lambda})\setminus\{0\}}\alpha^{-1}]$ by $W^{\lambda}$-invariance of $J_{\lambda}$. Therefore, $I\in\CH_{\lambda}\cap \CA_{\lambda}^{\prim}[\prod_{\alpha\in\CW(\Fg^{\lambda})\setminus\{0\}}\alpha^{-1}]=\CH_{\lambda}^{\prim}$.

For the second claim of the lemma, we have obviously $J_{\lambda}^{W^{\lambda}}\subset J_{\lambda}$. Combined with the first part of the lemma, this gives the reverse inclusion $J_{\lambda}^{W^{\lambda}}\subset J_{\lambda}\cap\CH_{\lambda}^{\prim}$. For the direct inclusion, if $f\in J_{\lambda}$ is also in $\CH_{\lambda}^{\prim}=(\CA_{\lambda}^{\prim})^{W^{\lambda}}$, it clearly is $W^{\lambda}$-invariant: $f\in J_{\lambda}^{W^{\lambda}}$.
\end{proof}

\begin{lemma}
\label{lemma:JlambdaWlambdainvariant}
 The submodule $J_{\lambda}$ is $W_{\lambda}$-invariant.
\end{lemma}
\begin{proof}
 Let $w\in W_{\lambda}$. We let $J_{\lambda}^w\coloneqq w\cdot J_{\lambda}$. Then, $J_{\lambda}^w$ is $W^{\lambda}$-stable by Lemma~\ref{lemma:Wlambdasubgroup}. It is an $\CA_{\lambda}^{\prim}$-submodule of $\CA_{\lambda}^{\prim}[\prod_{\alpha\in\CW(\Fg^{\lambda})\setminus\{0\}}(w\cdot\alpha)^{-1}]=\CA_{\lambda}^{\prim}[\prod_{\alpha\in\CW(\Fg^{\lambda})\setminus\{0\}}\alpha^{-1}]$. Moreover, it contains
 \[
  k_{w\cdot\mu,w\cdot\lambda}
 \]
for $\lambda\not\preceq\mu$. Now, we use that by Lemma~\ref{lemma:signkernelsim}, for $w\in W_{\lambda}$, $\{\pm k_{\mu,\lambda}\colon \lambda\not\preceq\mu\}=\{\pm k_{w\cdot\mu,w\cdot\lambda}\mid \lambda\not\preceq\mu\}$ since for $w\in W_{\lambda}$, $\overline{w\cdot\lambda}=\overline{\lambda}$. By minimality of $J_{\lambda}$, we deduce that $J_{\lambda}\subset J_{\lambda}^w$. By symmetry of the argument, we conclude $J_{\lambda}^w=J_{\lambda}$.
\end{proof}

\begin{corollary}
\label{corollary:JlambdaWlambdainv}
 $J_{\lambda}^{W^{\lambda}}$ is $W_{\lambda}$-invariant.
\end{corollary}
\begin{proof}
 This follows directly from Lemmas~\ref{lemma:Wlambdasubgroup} and~\ref{lemma:JlambdaWlambdainvariant}.
\end{proof}

 We let $\CP_{\lambda}\subset \CH_{\lambda}^{\prim}$ be a $W_{\lambda}$-stable direct sum complement of $J_{\lambda}^{W^{\lambda}}$ (using Corollary~\ref{corollary:JlambdaWlambdainv}):
 \[\CH_{\lambda}^{\prim}=J_{\lambda}^{W^{\lambda}}\oplus \CP_{\lambda}\,.
 \]

\begin{lemma}
\label{lemma:imageinduction}
 For any $\lambda\in X_*(T)$, the image of
 \[
  \bigoplus_{\lambda\not\preceq\mu}\CH_{\mu,\lambda}\xrightarrow{\bigoplus\Ind_{\mu,\lambda}}\CH_{\lambda}
 \]
is $J_{\lambda}^{W^{\lambda}}\otimes\HO^*(\pt/G_{\lambda})$.

Similarly, the image of
\[
 \bigoplus_{\mu\prec\lambda}\CH_{\mu,\lambda}\xrightarrow{\bigoplus\Ind_{\mu,\lambda}}\CH_{\lambda}
\]
is $J_{\lambda}^{W^{\lambda}}\otimes\HO^*(\pt/G_{\lambda})$.
\end{lemma}
\begin{proof}
We only prove the second part of the lemma, as the second part can be proved in the exact same way by using Lemma~\ref{lemma:characterisationJlambda}.

 This is a calculation very similar to that in the proof of Lemma~\ref{lemma:JlambdaWlambdainvariant}. Namely, if $f\in\CH_{\mu,\lambda}$ and $g\in\HO^*(\pt/G_{\lambda})$, $\Ind_{\mu,\lambda}(fg)=\Ind_{\mu,\lambda}(f)g$ by $W^{\lambda}$-invariance of $g$ (Lemma~\ref{lemma:W_lambdaG_lambda}). Moreover, $fg\in\CH_{\mu,\lambda}$. Therefore, a spanning subset of $J_{\lambda}^{W^{\lambda}}\otimes\HO^*(\pt/G_{\lambda})$ can be obtained in this way by considering all $\mu\prec\lambda$. Conversely, if $f\in\CH_{\mu,\lambda}\subset\CA_{\mu,\lambda}=\CA_{\lambda}=\HO^*(\pt/T)$ for $\mu\prec\lambda$, we write $f=\sum_{i=1}^Nf_i\otimes g_i$ with $f_i\in \CA_{\lambda}^{\prim}$ and $g_i\in\HO^*(\pt/G_{\lambda})$ so that
 \[
  \Ind_{\mu,\lambda}(f)=\sum_{i=1}^N\Ind_{\mu,\lambda}(f_i)\otimes g_i
 \]
 indeed belongs to $J_{\lambda}^{W^{\lambda}}\otimes\HO^*(\pt/G_{\lambda})$.
\end{proof}

\begin{lemma}
\label{lemma:surjective}
 For any $\lambda\in X_*(T)$, the map
 \[
  \bigoplus_{\substack{\mu\in X_*(T)\\\mu\preceq\lambda}}\CP_{\mu}\otimes\HO^*(\pt/G_{\mu})\xrightarrow{\bigoplus_{\mu\preceq\lambda}\Ind_{\mu,\lambda}} \CH_{\lambda}
 \]
is surjective.
\end{lemma}
\begin{proof}
 We proceed by induction on $\lambda\in X_*(T)$. If $\lambda$ is minimal, then $J_{\lambda}=\{0\}$ (by Lemma~\ref{lemma:characterisationJlambda}) and so $\CP_{\lambda}=\HO^*(V^{\lambda}/(G^{\lambda}/G_{\lambda}))$. By Proposition~\ref{proposition:decompositioncohomology}, the map
 \[
  \CP_{\lambda}\otimes\HO^*(\pt/G_{\lambda})\rightarrow\HO^*(V^{\lambda}/G^{\lambda})
 \]
of the lemma is indeed surjective (as it even is an isomorphism).

Now, let $\lambda\in X_*(T)$. We assume that the map of the lemma is surjective for any $\lambda'\prec\lambda$. By associativity of the induction (Proposition~\ref{proposition:associativity}), the image of the map of the lemma coincides with the image of the composition
\[
 \left(\bigoplus_{\mu\prec\lambda}\bigoplus_{\mu'\preceq\mu}\CP_{\mu'}\otimes\HO^*(\pt/G_{\mu'})\right)\oplus \left(\CP_{\lambda}\otimes\HO^*(\pt/G_{\lambda})\right)\xrightarrow{I\oplus\id} \left(\bigoplus_{\mu\prec\lambda}\CH_{\mu}\right)\oplus\left(\CP_{\lambda}\otimes\HO^*(\pt/G_{\lambda})\right)\rightarrow\CH_{\lambda}
\]
where the morphism $I$ is $\bigoplus_{\mu\prec\lambda}\bigoplus_{\mu'\preceq\mu}\Ind_{\mu',\mu}$.
By induction hypothesis,
\[
 \left(\bigoplus_{\mu\prec\lambda}\bigoplus_{\mu'\preceq\mu}\CP_{\mu'}\otimes\HO^*(\pt/G_{\mu'})\right)\rightarrow\left(\bigoplus_{\mu\prec\lambda}\CH_{\mu}\right)
\]
is surjective and by Lemma~\ref{lemma:imageinduction}, the image of 
\[
\left(\bigoplus_{\mu\prec\lambda}\CH_{\mu}\right)\rightarrow\CH_{\lambda}
\]
is $J_{\lambda}^{W^{\lambda}}\otimes\HO^*(\pt/G_{\lambda})$, of which $\CP_{\lambda}\otimes\HO^*(\pt/G_{\lambda})$ is a direct sum complement by definition of $\CP_{\lambda}$. This proves the surjectivity.
\end{proof}

 For $y\in\Ft$, we let $\Fg^{y}=\{\xi\in\Fg\mid [y,\xi]=0\}$ be the centraliser of $y$ in $\Fg$ and $V^{y}=\{v\in V\mid y\cdot v=0\}$ where $y$ acts on $V$ via the infinitesimal action.

\begin{lemma}
\label{lemma:cocharactertoelement}
For any $y\in\Ft$, there exists $\lambda\in X_*(T)$ such that $\Fg^{\lambda}=\Fg^{y}$ and $V^{\lambda}=V^{y}$.
\end{lemma}
\begin{proof}
 We may assume that $y\neq 0$. Via the natural embeddings $X_*(T)\subset \Ft_{\BoQ}$ and $X^*(T)\subset \Ft_{\BoQ}^*$, the natural pairing $X_*(T)\times X^*(T)\rightarrow \BoZ$ corresponds to the natural pairing $\langle-,-\rangle\colon\Ft\times\Ft^*\rightarrow\BoC$ (see also \S\ref{subsection:notations}). We let
 \[
  \CW_1\coloneqq \{\alpha\in\CW'(V)\cup\CW'(\Fg)\mid \langle y,\alpha\rangle=0\}
 \]
 and
\[
  \CW_2\coloneqq \{\alpha\in\CW'(V)\cup\CW'(\Fg)\mid \langle y,\alpha\rangle\neq0\}\,.
  \]
The rational vector space $\cap_{\alpha\in\CW_1}\alpha^{\perp}\subset\Ft_{\BoQ}$ is non-zero, as it is defined by the rational equations $\langle-,\alpha\rangle=0$ and its complexification contains $y$. The complement of the finite union of rational hyperplanes $U_y\coloneqq\cap_{\alpha\in\CW_1}\alpha^{\perp}\setminus\cup_{\alpha\in\CW_2}\alpha^{\perp}$ is therefore also non-empty. We let $y'\in U_{y}$ be a rational point ($y\in\Ft_{\BoQ}$). For some $N\geq 1$, $Ny'\in X_*(T)\subset \Ft$. We let $\lambda\coloneqq Ny'$. It satisfies the property of the lemma.
\end{proof}

\begin{proposition}
\label{proposition:finitedimensional}
 For any $\lambda\in X_*(T)$, $\CP_{\lambda}$ is a finite-dimensional vector space.
\end{proposition}
\begin{proof}
 By possibly replacing $V$ by $V^{\lambda}$ and $G$ by $G^{\lambda}$, we may assume that $\lambda=0$. It suffices to prove that $J_{\lambda}^{W^{\lambda}}\subset \CH_{\lambda}^{\prim}$ has finite codimension. If we replace $V$ by $V\times\Fg$, then the corresponding module $J_{\lambda}\subset\CA^{\prim}_{\lambda}[\prod_{\alpha\in\CW(\Fg)\setminus\{0\}}\alpha^{-1}]$ is smaller than the one for $V$ (by the very definition of $J_{\lambda}$ at the beginning of \S\ref{subsection:proofcohint}, since when replacing $V$ by $V\times\Fg$, the new $k_{\lambda}$ is equal to the numerator of the old $k_{\lambda}$), and therefore so is the corresponding module of invariants $J_{\lambda}^{W^{\lambda}}$, and therefore its codimension is bigger. Moreover, in this case, $J_{\lambda}\subset \CA_{\lambda}^{\prim}$, i.e. there is no need to localise.
 
 We have an inclusion
 \[
  \CH_{\lambda}^{\prim}/J_{\lambda}^{W^{\lambda}}\rightarrow\CA_{\lambda}^{\prim}/J_{\lambda}\,,
 \]
because $\CH_{\lambda}^{\prim}\cap J_{\lambda}=J_{\lambda}^{W^{\lambda}}$ (second part of Lemma~\ref{lemma:invariantsJH}). Therefore, it suffices to prove that the ideal $J_{\lambda}\subset\CA_{\lambda}^{\prim}$ has finite codimension. We see $\CA_{\lambda}^{\prim}$ as the ring of regular functions on the affine algebraic variety $\Ft/\Fg_{\lambda}$. We prove that the closed subscheme defined by $J_{\lambda}$ is supported on $\{0\}$. If the ideal $J_{\lambda}$ vanishes on $y\in(\Ft/\Fg_{\lambda})_{\overline{\BoQ}}\setminus\{0\}$ (where $\overline{\BoQ}$ is the field of algebraic numbers) we let $y'\in\Ft$ be a lift of $y$ and $\nu$ be given by Lemma~\ref{lemma:cocharactertoelement}. It satisfies $\CW^{\nu=0}(V)=\{\alpha\in\CW(V)\mid \langle y,\alpha\rangle=0\}\subset \CW(V)$ and $\CW^{\nu=0}(\Fg)=\{\alpha\in\CW(\Fg)\mid\langle y,\alpha\rangle=0\}\subset\CW(\Fg)$. Moreover, at least one of these inequalities is strict. Indeed, we have $y'\not\in\Fg_{\lambda}$, which means that either $y'$ is not in the center of $\Lie(G^{\lambda})$ and so $\nu$ defines a strict Levi $G^{\nu}=G^y$ of $G^{\lambda}$ or $y'$ is not in the Lie algebra of the kernel of the action $G^{\lambda}\rightarrow \GL(V^{\lambda})$ in which case $V^{\nu}=V^{y}\subsetneq V^{\lambda}$. Therefore, $\nu\prec\lambda=0$, i.e $\overline{\nu}$ is nontrivial. By definition, we have $\prod_{\alpha\in\CW^{\nu<0}(V^{\lambda})}\alpha\in J_{\lambda}$.
 
We also have $\prod_{\alpha\in\CW^{\nu<0}(V^{\lambda})}\langle y,\alpha\rangle\neq 0$: indeed, by choice of $\nu$, for any $\alpha\in\CW^{\nu<0}(V^{\lambda})$, $\langle\nu,\alpha\rangle=0$ if and only if $\langle y,\alpha\rangle=0$. This is a contradiction since we found an element $\prod_{\alpha\in\CW^{\nu<0}(V^{\lambda})}\alpha\neq 0$ of $J_{\lambda}$ not vanishing on $y$. Therefore, $J_{\lambda}$ defines a subscheme of $\Ft/\Fg_{\lambda}$ supported on $0$.
\end{proof}

The rest of this section is devoted to the proof of the injectivity of the cohomological integrality map in Theorem~\ref{theorem:cohintabsolute}.

For $\lambda\in X_*(T)$, we let $\tilde{J}_{\lambda}\coloneqq J_{\lambda}\CA_{\lambda}$. This is a $\CA_{\lambda}$-submodule of $\CA_{\lambda}^{\loc}$.

Recall the inclusion $\CH_{\lambda}=\CA_{\lambda}^{W^{\lambda}}\rightarrow \CA_{\lambda}$.

\begin{lemma}
\label{lemma:nonzerodivisor}
 The elements $\alpha\in\CW^{\lambda<0}(V)\cup\CW^{\lambda<0}(\Fg)$ (or more generally, $\alpha\in X^*(T)$ such that $\langle\lambda,\alpha\rangle\neq0$), seen as elements of $\CA_{\lambda}$, are not zero divisors in the quotient $\CA_{\lambda}$-module $\CA_{\lambda}^{\loc}/\tilde{J}_{\lambda}$.
\end{lemma}
\begin{proof}
 Let $\alpha\in\CW^{\lambda<0}(V)\cup\CW^{\lambda<0}(\Fg)$. We see $\alpha\in\Ft^*$ via the natural inclusion $X^*(T)\subset\Ft^*$ and $\langle\lambda,\alpha\rangle\neq 0$ implies that $\alpha$ does not vanish on $\Fg_{\lambda}$ since it takes a nonzero value on $\left(\frac{\mathrm{d}}{\mathrm{d}t}\lambda(t)\right)_{t=1}\in\Fg_{\lambda}\subset\Ft$. Therefore, the projection $\alpha'\in\Fg_{\lambda}^*$ of $\alpha$ with respect to the chosen (or any) decomposition $\Ft^*\cong(\Ft/\Fg_{\lambda})^*\oplus\Fg_{\lambda}^*$ is nonzero. 
 
 We choose a basis $(x_1,\hdots,x_{\dim\Fg_{\lambda}})$ of $\Fg_{\lambda}^*$ such that $x_1=\alpha'$ and we order the monomials by the lexicographic order so that $x_1\geq x_2\geq\hdots$. Let $g\in \CA_{\lambda}^{\loc}\cong\CA_{\lambda}^{\prim,\loc}\otimes\HO^*(\pt/G_{\lambda})$. It may be written in a unique way (see Lemma~\ref{lemma:G_lambdacohomology})
 \begin{equation}
 \label{equation:decompositiong}
  g=\sum_{(\nu_1,\hdots,\nu_{\dim\Fg_{\lambda}})\in\BoN^{\dim\Fg_{\lambda}}}g_{\nu}x^{\nu}
  \end{equation}
with $g_{\nu}\in\CA_{\lambda}^{\prim,\loc}$.
We have an equivalence between the following two statements:
\begin{enumerate}
 \item $g\not\in \tilde{J}_{\lambda}$,
 \item $g_{\nu}\not\in J_{\lambda}$ for some $\nu\in\BoN^{\dim\Fg_{\lambda}}$.
\end{enumerate}
The implication $(1)\implies(2)$ is immediate as if all $g_{\nu}$ belong to $J_{\lambda}$, then $g$ would be in $\tilde{J}_{\lambda}$. The reverse implication $(2)\implies(1)$ is also true by the unicity of the decomposition \eqref{equation:decompositiong} of $g$.

Assume that $g\not\in\tilde{J}_{\lambda}$. We need to show that for $\alpha\in\CW^{\lambda<0}(V)\cup\CW^{\lambda<0}(\Fg)$, $\alpha g\not\in\tilde{J}_{\lambda}$. We let $\nu=\max \{\nu'\in\BoN^{\dim\Fg_{\lambda}}\colon g_{\nu'}\neq 0\}$ (the maximum is taken with respect to the lexicographic order). We may assume that $g_{\nu}\not\in J_{\lambda}$.
We have
\[
\begin{aligned}
 \alpha g&=(\alpha-\alpha')g+\alpha'g\\
         &=(\alpha-\alpha')g+x_1g
\end{aligned}
\]
and $\alpha-\alpha'\in(\Ft/\Fg_{\lambda})^*\subset\Sym((\Ft/\Fg_{\lambda})^*)$. Therefore, the term of highest degree in the variables $(x_1,\hdots,x_{\dim\Fg_{\lambda}})$ in $\alpha g$ is $g_{\nu}x_1x^{\nu}$ and $g_{\nu}\not\in J_{\lambda}$. By the equivalence above, $\alpha g\not\in \tilde{J}_{\lambda}$. This concludes the proof.
\end{proof}

We let $\CA_{\lambda}'\coloneqq \CA_{\lambda}[\prod_{\alpha\in\CW^{\lambda<0}(V)\cup\CW^{\lambda<0}(\Fg)}\alpha^{-1}]$, $\CA'^{\loc}_{\lambda}\coloneqq\CA_{\lambda}[\alpha^{-1}\colon\alpha\in\CW(\Fg)\setminus\{0\}]$ and $\CH_{\lambda}'\coloneqq (\CA_{\lambda}')^{W^{\lambda}}$. We have the localisation maps
\[
 L\colon \CA_{\lambda}\rightarrow\CA_{\lambda}',\quad L\colon\CA_{\lambda}^{\loc}\rightarrow\CA_{\lambda}'^{\loc}, \quad L\colon \CH_{\lambda}\rightarrow\CH_{\lambda}'\,.
\]
We let $\tilde{J}_{\lambda}'\coloneqq \CA_{\lambda}'L(\tilde{J}_{\lambda})$ be the localised $\CA_{\lambda}'$-module inside $\CA_{\lambda}'^{\loc}$ (i.e. the $\CA_{\lambda}'$-submodule of $\CA_{\lambda}'^{\loc}$ generated by $J_{\lambda}$). We emphasized here the fact that $\tilde{J}'_{\lambda}$ is \emph{not} a $\CA'^{\loc}_{\lambda}$-submodule.

\begin{lemma}
\label{lemma:localizationinjective}
 The localisation maps
 \[
  L\colon \CA_{\lambda}^{\loc}/\tilde{J}_{\lambda}\rightarrow \CA_{\lambda}'^{\loc}/\tilde{J}_{\lambda}', \quad L\colon \CH_{\lambda}/(\tilde{J}_{\lambda})^{W^{\lambda}}\rightarrow \CH_{\lambda}'/(\tilde{J}_{\lambda}')^{W^{\lambda}}
 \]
are injective.
\end{lemma}
\begin{proof}
 This follows from Lemma~\ref{lemma:nonzerodivisor} as we localise by a set of nonzero divisors of the quotient $\CA_{\lambda}$-module $\CA_{\lambda}^{\loc}/\tilde{J}_{\lambda}$ in $\CA_{\lambda}$. More precisely, the first localization morphism $L\colon \CA_{\lambda}^{\loc}/\tilde{J}_{\lambda}\rightarrow \CA_{\lambda}'^{\loc}/\tilde{J}_{\lambda}'$ is the localization with respect to $\prod_{\alpha\in\CW^{\lambda<0}(V)\cup\CW^{\lambda<0}(\Fg)}\alpha$ for the $\CA_{\lambda}$-module $\CA_{\lambda}^{\loc}/\tilde{J}_{\lambda}$, and the localization element is not a zero-divisor of $\CA_{\lambda}^{\loc}/\tilde{J}_{\lambda}$. Therefore, the first localization morphism is injective.

 Now, the two localization morphisms of the lemma fit into the commutative diagram
\begin{equation}\label{equation:diagramlocalization}\begin{tikzcd}
	{\CH_{\lambda}/(\tilde{J}_{\lambda})^{W^{\lambda}}} & {\CH'_{\lambda}/(\tilde{J}'_{\lambda})^{W^{\lambda}}} \\
	{\CA_{\lambda}^{\loc}/\tilde{J}_{\lambda}} & {\CA'^{\loc}_{\lambda}/\tilde{J}'_{\lambda}} \\
	{}
	\arrow["L", from=1-1, to=1-2]
	\arrow["a"', from=1-1, to=2-1]
	\arrow[from=1-2, to=2-2]
	\arrow["L", from=2-1, to=2-2]
\end{tikzcd}\end{equation}
The vertical map $a$ is induced by the inclusion $\CH_{\lambda}=\CH_{\lambda}^{W^{\lambda}}\subset \CA_{\lambda}\subset\CA_{\lambda}^{\loc}$. The map $a$ is injective since $\CH_{\lambda}\cap\tilde{J}_{\lambda}=(\tilde{J}_{\lambda})^{W^{\lambda}}$, which follows from the arguments of the proof of Lemma~\ref{lemma:invariantsJH}. Therefore, the second localization map (the top line of \eqref{equation:diagramlocalization}) must be injective.
\end{proof}

\begin{lemma}
\label{lemma:productkernels}
 Let $\lambda, \mu\in X_{*}(T)$ and $\nu$ as in Lemma~\ref{lemma:gcdcocharacters}. Then,
  \[
   k_{\mu}=k_{\nu,\lambda}\frac{\prod_{\alpha\in\CW^{\mu<0}(V)\cap\CW^{\lambda\neq0}(V)}\alpha}{\prod_{\alpha\in\CW^{\mu<0}(\Fg)\cap\CW^{\lambda\neq0}(\Fg)}\alpha}\,.
 \]
\end{lemma}
\begin{proof}
 By the properties satisfied by the cocharacter $\nu$ given by Lemma~\ref{lemma:gcdcocharacters}, we have $\CW^{\mu<0}(V)=\CW^{\nu<0}(V)\cap\CW^{\lambda=0}(V)\sqcup \CW^{\mu<0}(V)\cap\CW^{\lambda\neq0}(V)$ and similarly for $V$ and $\Fg$ exchanged. Therefore,
 \[
 \begin{aligned}
  k_{\mu}&=\frac{\prod_{\alpha\in\CW^{\mu<0}(V)}\alpha}{\prod_{\alpha\in\CW^{\mu<0}(\Fg)}\alpha}\\
  &=\frac{\prod_{\alpha\in\CW^{\nu<0}(V)\cap\CW^{\lambda=0}(V)}\alpha}{\prod_{\alpha\in\CW^{\nu<0}(\Fg)\cap\CW^{\lambda=0}(\Fg)}\alpha}\frac{\prod_{\alpha\in\CW^{\mu<0}(V)\cap\CW^{\lambda\neq0}(V)}\alpha}{\prod_{\alpha\in\CW^{\mu<0}(\Fg)\cap\CW^{\lambda\neq0}(\Fg)}\alpha}\\
  &=k_{\nu,\lambda}\frac{\prod_{\alpha\in\CW^{\mu<0}(V)\cap\CW^{\lambda\neq0}(V)}\alpha}{\prod_{\alpha\in\CW^{\mu<0}(\Fg)\cap\CW^{\lambda\neq0}(\Fg)}\alpha}\,.
 \end{aligned}
 \]

\end{proof}

We define
\[
 k'_{\mu,\lambda}\coloneqq \frac{\prod_{\alpha\in\CW^{\mu<0}(V)\cap\CW^{\lambda\neq0}(V)}\alpha}{\prod_{\alpha\in\CW^{\mu<0}(\Fg)\cap\CW^{\lambda\neq0}(\Fg)}\alpha}\,.
\]
This is an invertible element in $\CA_{\lambda}'$.
\begin{lemma}
\label{lemma:idealkernel}
 Let $\lambda, \mu\in X_{*}(T)$ and $w$ in $W$ be such that $\lambda\not\preceq w^{-1}\cdot\mu$. Then, $w\cdot k_{\mu}\in \tilde{J}'_{\lambda}$.
\end{lemma}
\begin{proof}
 Of course, the presence of $W$ is superfluous since $w k_{\mu}=k_{w^{-1}\cdot \mu}$ for any $w\in W$ and it suffices to prove that for any $\lambda,\mu\in X_*(T)$ such that $\lambda\not\preceq\mu$, $k_{\mu}\in\tilde{J}'_{\lambda}$.
 
We let $\nu\in X_*(T)$ be such that $V^{\nu}=V^{\lambda}\cap V^{\mu}$, $\Fg^{\nu}=\Fg^{\lambda}\cap\Fg^{\mu}$, $(V^{\lambda})^{\nu\geq 0}=(V^{\lambda})^{\mu\geq 0}$ and $(\Fg^{\lambda})^{\nu\geq 0}=(\Fg^{\lambda})^{\mu\geq 0}$ (Lemma~\ref{lemma:gcdcocharacters}). By assumption, $\nu\prec\lambda$.

 We can write, by Lemma~\ref{lemma:productkernels},
 \[
   k_{\mu}=k_{\nu,\lambda}\underbrace{\frac{\prod_{\alpha\in\CW^{\mu<0}(V)\cap\CW^{\lambda\neq0}(V)}\alpha}{\prod_{\alpha\in\CW^{\mu<0}(\Fg)\cap\CW^{\lambda\neq0}(\Fg)}\alpha}}_{=k'_{\mu,\lambda}\in\CA'_{\lambda}}\,.
 \]
This element is in $\tilde{J}'_{\lambda}$ since $k_{\nu,\lambda}$ is in ${J}_{\lambda}$, by Lemma~\ref{lemma:characterisationJlambda}.
\end{proof}

\begin{lemma}
\label{lemma:stableWlambda}
 Let $\lambda\in X_*(T)$ be such that for some $w\in W$, $\lambda\preceq w\cdot\lambda$. Then, $\overline{w\cdot\lambda}=\overline{\lambda}$.
\end{lemma}
\begin{proof}
 The lemma follows from the inclusions $V^{\lambda}\subset V^{w\cdot\lambda}$ and $\Fg^{\lambda}\subset\Fg^{w\cdot\lambda}$ and the equalities of dimensions $\dim V^{\lambda}=\dim V^{w\cdot\lambda}$ and $\dim\Fg^{\lambda}=\dim\Fg^{w\cdot\lambda}$.
\end{proof}

For $\lambda\in X_*(T)$, we have an inclusion
\[
 \Res_{\lambda}\colon \CH\cong\CA^{W}\cong\CA_{\lambda}^{W}\rightarrow\CH_{\lambda}\cong \CA_{\lambda}^{W^{\lambda}}\,.
\]
since $\CA\cong\CA_{\lambda}=\HO^*(\pt/T)$ and $W^{\lambda}\subset W$ (Lemma~\ref{lemma:Wlambdasubgroup}).

\begin{corollary}
\label{corollary:congruence}
For any $f_{\mu}\in\CH_{\mu}$, $\tilde{\mu}\in\SP_V/W$ such that $f_{\mu}=0$ if $\lambda\prec w\cdot\mu$ for some $w\in W$, we have
 \[
  L\Res_{\lambda}\left(\sum_{\tilde{\mu}\in\SP_V/W}\Ind_{\mu}(f_{\mu})\right)\equiv k_{\lambda}\cdot\sum_{w\in W_{\lambda}/W^{\lambda}}\varepsilon_{V,\lambda}(w)^{-1}(w\cdot f_{\lambda})\pmod{\tilde{J}'_{\lambda}}\,.
 \]
\end{corollary}
\begin{proof}
 We calculate:
 \[
 \begin{aligned}
  L\Res_{\lambda}\left(\sum_{\tilde{\mu}\in\SP_V/W}\Ind_{\mu}(f_{\mu})\right)&=L\Res_{\lambda}\left( \sum_{\tilde{\mu}\in\SP_V/W}\frac{1}{\# W^{\mu}}\sum_{w\in W}(w\cdot f_{\mu})(w\cdot k_{\mu})\right)\quad\text{by Proposition~\ref{proposition:explicitformulainduction}}\\
  &\equiv\frac{1}{\# W^{\lambda}}\sum_{\substack{w\in W\\ w\overline{\lambda}=\overline{\lambda}}}(w\cdot f_{\lambda})(w\cdot k_{\lambda})\pmod{\tilde{J}_{\lambda}'}\quad\text{by Lemmas~\ref{lemma:idealkernel} and \ref{lemma:stableWlambda}}\\
  &= \frac{1}{\# W^{\lambda}}k_{\lambda}\cdot\sum_{w\in W_{\lambda}}\varepsilon_{V,\lambda}(w)^{-1}(w\cdot f_{\lambda})\quad\text{by definition of the character $\varepsilon_{V,\lambda}$, Proposition~\ref{proposition:characterW_lambda}}\,.
  \end{aligned}
 \]
\end{proof}

\begin{proposition}
\label{proposition:injectivity}
 The map
 \[
  \bigoplus_{\tilde{\mu}\in\SP_V/W}(\CP_{\mu}\otimes\HO^*(\pt/G_{\mu}))^{\varepsilon_{V,\mu}}\xrightarrow{\bigoplus\Ind_{\mu}}\HO^*(V/G)
 \]
induced by the parabolic induction morphisms is injective.
\end{proposition}
\begin{proof}
Let $f_{\mu}\in(\CP_{\mu}\otimes\HO^*(\pt/G_{\mu}))^{\varepsilon_{V,\mu}}$, $\tilde{\mu}\in\SP_V/W$. We let $U\coloneqq \sum_{\tilde{\mu}\in\SP_V/W}\Ind_{\mu}(f_{\mu})$. We assume that not all of the $f_{\mu}$'s are zero, and we have to prove that $U\neq0$. We let $\tilde{\lambda}\in\SP_V/W$ be such that $f_{\lambda}\neq0$ and for any $\tilde{\mu}\in\SP_V/W$ such that $f_{\mu}\neq0$, for any $w\in W$, $\lambda\not\preceq w\cdot \mu$ (otherwise, we may replace $\tilde{\lambda}$ by $\tilde{\mu}$, etc.).

By Corollary~\ref{corollary:congruence}, we have
 \[
 \begin{aligned}
  L\Res_{\lambda}\left(\sum_{\tilde{\mu}\in\SP_V/W}\Ind_{\mu}(f_{\mu})\right)&\equiv k_{\lambda}\cdot\sum_{w\in W_{\lambda}/W^{\lambda}}\varepsilon_{V,\lambda}(w)^{-1}(w\cdot f_{\lambda})\pmod{\tilde{J}'_{\lambda}}\\ &\equiv k_{\lambda}\# (W_{\lambda}/W^{\lambda}) f_{\lambda}\pmod{\tilde{J}'_{\lambda}}\quad\text{by Proposition \ref{proposition:projector}}\,.
  \end{aligned}
 \]
and moreover, $f_{\lambda}\not\in\tilde{J}'_{\lambda}$ since $f_{\lambda}\in\CP_{\lambda}\otimes\HO^*(\pt/G_{\lambda})$, and $k_{\lambda}$ is invertible in $\CA_{\lambda}'$. We obtain that $\Res_{\lambda}(U)\neq0$. Therefore, $U\neq 0$, proving injectivity.
\end{proof}

We are now ready to prove Theorem~\ref{theorem:cohintabsolute}.
\begin{corollary}
\label{corollary:cohintiso}
 The map \[
  \bigoplus_{\tilde{\mu}\in\SP_V/W}(\CP_{\mu}\otimes\HO^*(\pt/G_{\mu}))^{\varepsilon_{V,\mu}}\xrightarrow{\bigoplus_{\tilde{\mu}\in\SP_V/W}\Ind_{\mu}}\HO^*(V/G)
 \]
 is an isomorphism.
\end{corollary}
\begin{proof}
 The injectivity is Proposition~\ref{proposition:injectivity}. The surjectivity comes from Lemmas~\ref{lemma:surjective} and~\ref{lemma:sameimageWeylgroup}. Indeed, by Lemma~\ref{lemma:surjective}, any element $f\in\HO^*(V/G)$ has the form $f=\sum_{\mu\preceq 0}\Ind_{\mu}(g_{\mu})$ for some $g_{\mu}\in\CP_{\mu}\otimes\HO^*(\pt/G_{\mu})$. By Lemma~\ref{lemma:sameimageWeylgroup}, we may rewrite $f=\sum_{\tilde{\mu}\in\SP_{V}/W}\Ind_{\mu}(\tilde{g}_{\mu})$. We have the formula
 \[
  \Ind_{\mu}(\tilde{g}_{\mu})=\frac{1}{\# W^{\mu}}\sum_{\overline{w'}\in W/W_{\mu}}w'\cdot\left(\sum_{w\in W_{\mu}}\varepsilon_{V,\mu}(w)^{-1}w\cdot \tilde{g}_{\mu} \right)(w'\cdot k_{\mu})
 \]
for $\tilde{g}_{\mu}\in\CH_{\mu}$, deduced from Propositions~\ref{proposition:explicitformulainduction} and~\ref{proposition:characterW_lambda} which implies, using Proposition~\ref{proposition:projector}, that $\Ind_{\mu}(\tilde{g}_{\mu})$ vanishes if $f$ is not in the $\varepsilon_{V,\mu}$-isotypic component of $\CP_{\mu}\otimes\HO^*(\pt/G_{\mu})$ for the $W_{\mu}$-action. Therefore, we may assume that $\tilde{g}_{\mu}$ is $\varepsilon_{V,\lambda}$-invariant. This concludes the proof of the corollary.
\end{proof}

\subsection{The BPS subspaces}
In analogy to the case of quivers \cite{efimov2012cohomological,meinhardt2019donaldson} and also \cite{davison2020cohomological} for quivers with potential, we may call the subspaces $\CP_{\lambda}\subset\HO^*(V^{\lambda}/G^{\lambda})$ the \emph{BPS subspaces}. Although we lack both qualitative and quantitative properties of these spaces, there are cases when we can prove they vanish.

\begin{proposition}
 Let $\lambda\in X_*(T)$. If $\ker(G^{\lambda}\rightarrow\GL(V^{\lambda}))^{\circ}\not\subset Z(G^{\lambda})$, then $\CP_{\lambda}=\{0\}$.
\end{proposition}
\begin{proof}
 By possibly replacing $V$ by $V^{\lambda}$ and $G$ by $G^{\lambda}$, we may assume that $\lambda$ is the trivial cocharacter of $T$. If $\ker(G\rightarrow\GL(V))^{\circ}\not\subset Z(G)$, then there exists a cocharacter $\mu\colon\BoG_{\rmm}\rightarrow \ker(G\rightarrow\GL(V))\subset T$ which does not factor through $Z(G)$. We have $\overline{\mu}\in\SP_V\setminus\{\overline{0}\}$. Then, the induction diagram reads
\[\begin{tikzcd}
	& {V/P_{\mu}} \\
	{V/G^{\mu}} && {V/G}
	\arrow["{q_{\mu}}"', from=1-2, to=2-1]
	\arrow["{p_{\mu}}", from=1-2, to=2-3]
\end{tikzcd}\]
and the induction map $\Ind_{\mu}\colon \HO^*(V/G^{\mu})=\HO^*(\pt/T)^{W^{\mu}}\rightarrow \HO^*(V/G)=\HO^*(\pt/T)$ is the map
\[
 f\mapsto \sum_{w\in W/W^{\mu}}w\cdot\left(\frac{f}{\prod_{\alpha\in\CW^{\mu<0}(\Fg)}\alpha}\right)\,.
\]
It is surjective. Hence, $\CP_0=\{0\}$.
\end{proof}

\section{Vanishing bounds for the BPS spaces}
We recall that $\CH_{\lambda}=\HO^*(V^{\lambda}/G^{\lambda})$ for $\lambda\in X_*(T)$. We recall the chosen $\BoZ$-grading on $\CH_{\lambda}$: $\CH_{\lambda}^k=\HO^{k+d_{\lambda}}(V^{\lambda}/G^{\lambda})$ for $k\in\BoZ$. This is just a shift of the cohomological grading, chosen so that induction morphisms preserve this grading (Lemma~\ref{lemma:preservescohdegrees}). In this section, we prove Proposition~\ref{proposition:bounds}.

\begin{proof}[Proof of the lower bound in Proposition~\ref{proposition:bounds}]
First, we have an inclusion $\CP_{\lambda}^k\subset \CH_{\lambda}^k=\HO^{k+d_{\lambda}}(V^{\lambda}/G^{\lambda})$, and the latter vector space is zero for $k<-d_{\lambda}=\dim G^{\lambda}-\dim V^{\lambda}$. 
\end{proof}

The proof of the upper bound uses \cite[Lemma~3.6]{efimov2012cohomological} and is analogous to the proof of \cite[Theorem~3.5]{efimov2012cohomological}.

\begin{lemma}[{\cite[Lemma~3.6]{efimov2012cohomological}}]
\label{lemma:efimovlemma}
 Let $k$ be a field and let $B=k[z_1,\hdots,z_n]$ be an algebra of polynomials with $\deg(z_i)=2$ (to match the cohomological degrees). Suppose that $\ell_1,\hdots,\ell_s\in B^1$ are pairwise linearly independent linear forms in $z_i$. Let $\{P_1,\hdots,P_r\}\subset B$ be  non-empty set of polynomials of the form
 \[
  P_i=\prod_{j=1}^s\ell_j^{d_{i,s}}
 \]
for some $d_{i,j}\in\BoZ_{\geq0}$. Let $d_j\coloneqq \max_{1\leq i\leq r}d_{i,j}$ for $1\leq j\leq s$. Then, the following are equivalent:
\begin{enumerate}[(i)]
 \item $B^{2d}\subset (P_1,\hdots,P_r)$ for $d\geq \sum_{j=1}^sd_j-n+1$,
 \item the ideal $(P_1,\hdots,P_r)\subset B$ has finite codimension,
 \item For any sequence $p_1,\hdots,p_r$ of numbers in $\{1,\hdots,s\}$ such that $d_{i,p_i}>0$ for $1\leq i\leq r$, the linear forms generate the space $B^1$.
\end{enumerate}

\end{lemma}

\begin{proof}[Proof of the upper bound in Proposition~\ref{proposition:bounds}]
As we already noticed in the proof of Proposition \ref{proposition:finitedimensional}, if we replace $V$ by $V\times \Fg$, then we have inclusions $\CP_{V,\lambda}\subset \CP_{V\times\Fg,\lambda}$, and therefore, it suffices to give the upper bound for $\CP_{V\times\Fg,\lambda}$. We set $B\coloneqq \Sym[(\Ft/\Fg_{\lambda})^*]=\CA_{\lambda}^{\prim}$, the (polynomial) algebra of functions on $\Fg/\Fg_{\lambda}$. The proof of Proposition \ref{proposition:finitedimensional} shows that the ideal $J_{\lambda}\subset\Sym[(\Ft/\Fg_{\lambda})^*]$ has finite codimension, that is satisfies the second point of Lemma~\ref{lemma:efimovlemma}. In our case, the quantity $\sum_{j=1}^sd_j$ is equal to $\frac{\dim V^{\lambda}-\dim V^T}{2}$. Therefore, an upper bound for $\sum_{j=1}^sd_j-n+1$ is $\frac{\dim V^{\lambda}}{2}$.

Now, Lemma~\ref{lemma:efimovlemma} gives $(\CA_{\lambda}^{\prim})^{2d}\subset J_{\lambda}$ for $d\geq \frac{\dim V^{\lambda}}{2}$ (for the unshifted cohomological gradings), where $(\CA_{\lambda}^{\prim})^e$ denotes degree $e$ polynomials inside $\CA_{\lambda}^{\prim}$. By taking into account the cohomological shift by $\dim G^{\lambda}-\dim V^{\lambda}$, we obtain the upper bound: $\CP_{\lambda}^k=0$ for $k>\dim V^{\lambda}+\dim G^{\lambda}-\dim V^{\lambda}=\dim G^{\lambda}$.
\end{proof}

\section{Examples}
In this section, we give some explicit examples of cohomological integrality isomorphisms for some choices of pairs $(G,V)$ of a reductive group $G$ and a (weakly) symmetric representation $V$ of $G$. We give each time the explicit formula for the cohomological integrality isomorphisms produced by Theorem~\ref{theorem:cohintabsolute}. In each case, it is possible to verify by hand that they are indeed isomorphisms.
\subsection{$(\BoC^*)^2\curvearrowright\BoC^2\oplus(\BoC^2)^{\vee}$}
We let $G\coloneqq(\BoC^*)^2$ act on $V\coloneqq \BoC^2\oplus(\BoC^2)^{\vee}\cong \BoC^4$ by
\[
 (t,u)\cdot (a,b,c,d)\coloneqq (ta,ub,t^{-1}c,u^{-1}d)\,.
\]

We can identify the set of equivalence classes of cocharacters (\S\ref{subsection:parindsymreps}) $\SP_{V}=\{\overline{\lambda_0},\overline{\lambda_1},\overline{\lambda_2},\overline{\lambda_3}\}$ where
\[
 \begin{matrix}
  \lambda_0&\colon&\BoG_{\rmm}&\rightarrow&(\BoC^*)^2\\
  &&v&\mapsto&(1,1)\,,
 \end{matrix}
\]
\[
 \begin{matrix}
  \lambda_1&\colon&\BoG_{\rmm}&\rightarrow&(\BoC^*)^2\\
  &&v&\mapsto&(v^{-1},1)\,,
 \end{matrix}
\]
\[
 \begin{matrix}
  \lambda_2&\colon&\BoG_{\rmm}&\rightarrow&(\BoC^*)^2\\
  &&v&\mapsto&(1,v^{-1})\,,
 \end{matrix}
\]
\[
 \begin{matrix}
  \lambda_3&\colon&\BoG_{\rmm}&\rightarrow&(\BoC^*)^2\\
  &&v&\mapsto&(v^{-1},v^{-1})\,.
 \end{matrix}
\]
We have $V^{\lambda_0}=V$, $V^{\lambda_1}=(0\oplus\BoC)\oplus(0\oplus\BoC^{\vee})$, $V^{\lambda_2}=(\BoC\oplus0)\oplus(\BoC^{\vee}\oplus0)$ and $V^{\lambda_3}=\{0\}$. We have $G^{\lambda_0}=G^{\lambda_1}=G^{\lambda_2}=G^{\lambda_3}=(\BoC^*)^2$. We can describe the induction maps as follows:
\[
\begin{matrix}
  \Ind_{\lambda_1,\lambda_0}&\colon& \overbrace{\BoQ[x_1,x_2]}^{=\HO^*(V^{\lambda_1}/G^{\lambda_1})}&\rightarrow&\overbrace{\BoQ[x_1,x_2]}^{=\HO^*(V/G)}\\
&&f&\mapsto&x_1f
\end{matrix}
\]
\[
\begin{matrix}
  \Ind_{\lambda_2,\lambda_0}&\colon& \overbrace{\BoQ[x_1,x_2]}^{=\HO^*(V^{\lambda_2}/G^{\lambda_2})}&\rightarrow&\overbrace{\BoQ[x_1,x_2]}^{=\HO^*(V/G)}\\
&&f&\mapsto&x_2f
\end{matrix}
\]
\[
\begin{matrix}
  \Ind_{\lambda_3,\lambda_0}&\colon& \overbrace{\BoQ[x_1,x_2]}^{=\HO^*(V^{\lambda_3}/G^{\lambda_3})}&\rightarrow&\overbrace{\BoQ[x_1,x_2]}^{\HO^*(V/G)}\\
&&f&\mapsto&x_1x_2f
\end{matrix}
\]
\[
\begin{matrix}
  \Ind_{\lambda_3,\lambda_1}&\colon& \overbrace{\BoQ[x_1,x_2]}^{=\HO^*(V^{\lambda_3}/G^{\lambda_3})}&\rightarrow&\overbrace{\BoQ[x_1,x_2]}^{=\HO^*(V^{\lambda_1}/G^{\lambda_1})}\\
&&f&\mapsto&x_2f
\end{matrix}
\]
\[
\begin{matrix}
  \Ind_{\lambda_3,\lambda_2}&\colon& \overbrace{\BoQ[x_1,x_2]}^{=\HO^*(V^{\lambda_3}/G^{\lambda_3})}&\rightarrow&\overbrace{\BoQ[x_1,x_2]}^{=\HO^*(V^{\lambda_2}/G^{\lambda_2})}\\
&&f&\mapsto&x_1f
\end{matrix}
\]
We have
\[
 \CP_{\lambda_0}=\BoQ\subset \BoQ[x_1,x_2],\quad \CP_{\lambda_1}=\BoQ\subset \BoQ[x_1,x_2],\quad \CP_{\lambda_2}=\BoQ\subset \BoQ[x_1,x_2],\quad \CP_{\lambda_3}=\BoQ\subset \BoQ[x_1,x_2]\,.
\]
We have
\[
 G_{\lambda_0}=1\subset(\BoC^*)^2,\quad G_{\lambda_1}=\BoC^*\times 1\subset(\BoC^*)^2,\quad G_{\lambda_2}=1\times\BoC^*\subset(\BoC^*)^2,\quad G_{\lambda_3}=(\BoC^*)^2\,.
\]

The cohomological integrality isomorphism is
\[
\begin{matrix}
  \overbrace{\BoQ}^{=\CP_{\lambda_0}}\oplus \overbrace{\BoQ[x_1]}^{=\CP_{\lambda_1}\otimes\HO^*(\pt/G_{\lambda_1})}\oplus\overbrace{\BoQ[x_2]}^{=\CP_{\lambda_2}\otimes\HO^*(\pt/G_{\lambda_2})}\oplus\overbrace{\BoQ[x_1,x_2]}^{=\CP_{\lambda_3}\otimes\HO^*(\pt/G_{\lambda_3})}&\rightarrow&\overbrace{\BoQ[x_1,x_2]}^{=\HO^*(V/G)}\\
(a_0,f(x_1),g(x_2),h(x_1,x_2))&\mapsto&a_0+x_1f(x_1)+x_2g(x_2)+x_1x_2h(x_1,x_2)\,.
\end{matrix}
\]

\subsection{$\GL_2(\BoC)\curvearrowright\BoC^2\oplus(\BoC^2)^{\vee}$}
\label{subsection:GL2}
We let $G\coloneqq\GL_2(\BoC)$ act on $V\coloneqq\BoC^2\oplus(\BoC^2)^{\vee}$ via
\[
 g\cdot (u,v)\coloneqq (gu,(g^{-1})^tv)\,,
\]
where the superscript $t$ indicates the transpose. We let $T\coloneqq (\BoC^*)^2\subset\GL_2(\BoC)$ be the standard torus. The Weyl group is $W=\FS_2$. We have $\SP_V/W=\{\tilde{\lambda_0},\tilde{\lambda_1},\tilde{\lambda_2}, \tilde{\lambda}_3\}$ where
\[
 \begin{matrix}
  \lambda_0&\colon&\BoG_{\rmm}&\rightarrow&(\BoC^*)^2\\
  &&v&\mapsto&(1,1)
 \end{matrix}
\]
\[
 \begin{matrix}
  \lambda_1&\colon&\BoG_{\rmm}&\rightarrow&(\BoC^*)^2\\
  &&v&\mapsto&(v^{-1},1)
 \end{matrix}
\]
\[
 \begin{matrix}
  \lambda_2&\colon&\BoG_{\rmm}&\rightarrow&(\BoC^*)^2\\
  &&v&\mapsto&(v^{-1},v^{-2})
 \end{matrix}
\]
\[
 \begin{matrix}
  \lambda_3&\colon&\BoG_{\rmm}&\rightarrow&(\BoC^*)^2\\
  &&v&\mapsto&(v^{-1},v^{-1})
 \end{matrix}
\]

Therefore, we have $V^{\lambda_0}=V$, $V^{\lambda_1}=(0\oplus\BoC)\oplus(0\oplus\BoC^{\vee})$, $V^{\lambda_2}=\{0\}$, and $V^{\lambda_3}=\{0\}$. We have $G^{\lambda_0}=\GL_2(\BoC), G^{\lambda_1}=T$, $G^{\lambda_2}=T$ and $G^{\lambda_3}=G$. We can describe the induction maps as follows:
\[
\begin{matrix}
  \Ind_{\lambda_1,\lambda_0}&\colon& \overbrace{\BoQ[x_1,x_2]}^{=\HO^*(V^{\lambda_1}/G^{\lambda_1})}&\rightarrow&\overbrace{\BoQ[x_1,x_2]^{\FS_2}}^{=\HO^*(V/G)}=\BoQ[x_1+x_2,x_1x_2]\\
&&f&\mapsto&\frac{x_1}{x_1-x_2}f(x_1,x_2)+\frac{x_2}{x_2-x_1}f(x_2,x_1)
\end{matrix}
\]
\[
\begin{matrix}
  \Ind_{\lambda_2,\lambda_0}&\colon& \overbrace{\BoQ[x_1,x_2]}^{=\HO^*(V^{\lambda_2}/G^{\lambda_2})}&\rightarrow&\overbrace{\BoQ[x_1,x_2]^{\FS_2}}^{=\HO^*(V/G)}=\BoQ[x_1+x_2,x_1x_2]\\
&&f&\mapsto&\frac{x_1x_2}{x_1-x_2}f(x_1,x_2)+\frac{x_1x_2}{x_2-x_1}f(x_2,x_1)
\end{matrix}
\]
\[
\begin{matrix}
  \Ind_{\lambda_2,\lambda_1}&\colon& \overbrace{\BoQ[x_1,x_2]}^{=\HO^*(V^{\lambda_2}/G^{\lambda_2})}&\rightarrow&\overbrace{\BoQ[x_1,x_2]}^{\HO^*(V^{\lambda_1}/G^{\lambda_1})}\\
&&f&\mapsto&x_2f
\end{matrix}
\]
\[
\begin{matrix}
  \Ind_{\lambda_2,\lambda_3}&\colon& \overbrace{\BoQ[x_1,x_2]}^{=\HO^*(V^{\lambda_2}/G^{\lambda_2})}&\rightarrow&\overbrace{\BoQ[x_1+x_2,x_1x_2]}^{\HO^*(V^{\lambda_3}/G^{\lambda_3})}\\
&&f&\mapsto&\frac{f(x_1,x_2)-f(x_2,x_1)}{x_1-x_2}
\end{matrix}
\]

The induction $\Ind_{\lambda_1,\lambda_0}$ is surjective, since
\[
 x_1+x_2=\Ind_{\lambda_1,\lambda_0}(x_1)
\]
and
\[
 x_1x_2=\Ind_{\lambda_1,\lambda_0}\left(\frac{1}{2}x_1x_2\right)\,.
\]

The induction $\Ind_{\lambda_2,\lambda_3}$ is surjective as any symmetric polynomial $f$ is the image by $\Ind_{\lambda_2,\lambda_3}$ of $(x_1-x_2)f$.

A direct sum complement of the image of $\Ind_{\lambda_2,\lambda_1}$ is $\BoQ[x_1]=\CP_{\lambda_1}\otimes \BoQ[x_1]$ by definition of $\CP_{\lambda_1}$.

Therefore, we have $\CP_{\lambda_0}=0$, $\CP_{\lambda_1}=\BoQ$, $\CP_{\lambda_2}=\BoQ$ and $\CP_{\lambda_3}=0$. We have $G_{\lambda_0}=1, G_{\lambda_1}=\BoC^*\times 1$, $G_{\lambda_2}=T$ and $G_{\lambda_3}=G$. The cohomological integrality isomorphism reads
\[
\begin{matrix}
  \overbrace{\BoQ[x_1]}^{=\CP_{\lambda_1}\otimes\HO^*(\pt/G_{\lambda_1})}\oplus\overbrace{\BoQ[x_1,x_2]^{\sgn}}^{=(\CP_{\lambda_2}\otimes\HO^*(\pt/G_{\lambda_2}))^{\varepsilon_{V,\lambda_2}}}&\rightarrow&\overbrace{\BoQ[x_1,x_2]^{\FS_2}}^{=\HO^*(V/G)}\\
(f(x_1),g(x_1,x_2))&\mapsto&\frac{x_1f(x_1)-x_2f(x_2)}{x_1-x_2}+2x_1x_2\frac{g(x_1,x_2)}{x_1-x_2}
\end{matrix}
\]
where $\BoQ[x_1,x_2]^{\sgn}$ is the sign-isotypic component, since the kernel $\frac{x_1x_2}{x_1-x_2}$ changes sign when we exchange the variables $x_1,x_2$, and so the character $\varepsilon_{V,\lambda_2}$ of $W_{\lambda_2}=\mathfrak{S}_2$ given by Proposition~\ref{proposition:characterW_lambda} is the sign.

\subsection{$\GL_2(\BoC)\curvearrowright(\BoC^2\oplus(\BoC^2)^{\vee})^g$}
We generalize the previous situation by considering the diagonal action of $\GL_2(\BoC)$ on $(\BoC^2\oplus(\BoC^2)^{\vee})^g$ ($g\geq 0$). The set $\SP_{V}/W$ is the same as in \S\ref{subsection:GL2}. The fixed point sets $V^{\lambda_0}, V^{\lambda_1}$ and $V^{\lambda_2}$ can be computed similarly and the Levi subgroups $G^{\lambda_0}, G^{\lambda_1}$ and $G^{\lambda_2}$ are the ones given in \S\ref{subsection:GL2}. We can describe the induction maps as follows, analogously to \S\ref{subsection:GL2}.
\[
\begin{matrix}
  \Ind_{\lambda_1,\lambda_0}&\colon& \BoQ[x_1,x_2]&\rightarrow&\BoQ[x_1,x_2]^{\FS_2}=\BoQ[x_1+x_2,x_1x_2]\\
&&f&\mapsto&\frac{x_1^g}{x_1-x_2}f(x_1,x_2)+\frac{x_2^g}{x_2-x_1}f(x_2,x_1)
\end{matrix}
\]
\[
\begin{matrix}
  \Ind_{\lambda_2,\lambda_0}&\colon& \BoQ[x_1,x_2]&\rightarrow&\BoQ[x_1,x_2]^{\FS_2}=\BoQ[x_1+x_2,x_1x_2]\\
&&f&\mapsto&\frac{x_1^gx_2^g}{x_1-x_2}f(x_1,x_2)+\frac{x_1^gx_2^g}{x_2-x_1}f(x_2,x_1)
\end{matrix}
\]
\[
\begin{matrix}
  \Ind_{\lambda_2,\lambda_1}&\colon& \BoQ[x_1,x_2]&\rightarrow&\BoQ[x_1,x_2]\\
&&f&\mapsto&x_2^gf
\end{matrix}
\]
\[
\begin{matrix}
  \Ind_{\lambda_2,\lambda_3}&\colon&\BoQ[x_1,x_2]&\rightarrow&\BoQ[x_1+x_2,x_1x_2]\\
&&f&\mapsto&\frac{f(x_1,x_2)-f(x_2,x_1)}{x_1-x_2}
\end{matrix}
\]

We have $\CP_{\lambda_0}=\bigoplus_{j=0}^{g-2}\BoQ(x_1+x_2)^{j}$, $\CP_{\lambda_1}=\bigoplus_{j=0}^{g-1}\BoQ x_2^j$, $\CP_{\lambda_2}=\BoQ$ and $\CP_{\lambda_3}=0$ and the cohomological integrality isomorphism reads
\[
\begin{matrix}
  \CP_{\lambda_{0}}\oplus (\CP_{\lambda_1}\otimes\BoQ[x_1])\oplus\BoQ[x_1,x_2]^{\sgn}&\rightarrow&\BoQ[x_1,x_2]^{\FS_2}\\
(f(x_1,x_2),g(x_1,x_2),h(x_1,x_2))&\mapsto&f(x_1,x_2)+\frac{x_1^gg(x_1,x_2)-x_2^gg(x_2,x_1)}{x_1-x_2}+2x_1^gx_2^g\frac{h(x_1,x_2)}{x_1-x_2}\,.
\end{matrix}
\]

\subsection{$\rmSL_2(\BoC)\curvearrowright\BoC^2$}
\label{subsection:SL2C2}

We let $G\coloneqq\rmSL_2(\BoC)$ act on $\BoC^2$ via the natural action. It is a symmetric representation. We have $W=\FS_2$. We let $T\subset\rmSL_2(\BoC)$ be the maximal torus of diagonal matrices. We have $\SP_V/W=\{\tilde{\lambda_0},\tilde{\lambda_1}\}$ where
\[
 \begin{matrix}
  \lambda_0&\colon&\BoG_{\rmm}&\rightarrow&T\\
  &&v&\mapsto&(1,1)
 \end{matrix}
\]
\[
 \begin{matrix}
  \lambda_1&\colon&\BoG_{\rmm}&\rightarrow&T\\
  &&v&\mapsto&(v^{-1},v)\,.
 \end{matrix}
\]
We have $V^{\lambda_0}=V$, $V^{\lambda_1}=\{0\}$. We can describe the induction map as follows.
\[
\begin{matrix}
  \Ind_{\lambda_1,\lambda_0}&\colon& \BoQ[x]&\rightarrow&\BoQ[x^2]\\
&&f&\mapsto&\frac{1}{2}(f(x)+f(-x))\,.
\end{matrix}
\]
The factor $\frac{1}{2}$ comes from the fact that the $T$-weight of $\Fg^{\lambda_1>0}=\begin{pmatrix}
0&*\\
0&0
                        \end{pmatrix}
$ is $2$. We have $\CP_{\lambda_1}=\BoQ$ and $\CP_{\lambda_0}=\{0\}$ and the cohomological integrality isomorphism is
\[
 \begin{matrix}
  \BoQ[x^2]&\rightarrow&\BoQ[x^2]\\
  f&\mapsto&f\,.
 \end{matrix}
\]

\subsection{$\rmSL_2(\BoC)\curvearrowright \BoC^d$}
\label{subsection:SL2Cd}

Let $V_d$ be the $d$-dimensional irreducible representation of $\rmSL_2(\BoC)$. Up to a scaling factor $C_{d}$, the induction morphism is
\[
\begin{matrix}
  \Ind_{\lambda_1,\lambda_0}&\colon& \BoQ[x]&\rightarrow&\BoQ[x^2]\\
&&f&\mapsto&C_{d}x^{\lfloor\frac{d}{2}\rfloor-1}f(x)+C_{d}(-x)^{\lfloor\frac{d}{2}\rfloor-1}f(-x)\,.
\end{matrix}
\]
We have $\CP_{\lambda_1}=\BoQ$ and $\CP_{\lambda_0}=\BoQ[x^2]_{\deg<2(\lfloor\frac{d}{2}\rfloor-1)}$
where $C_{d}\in\BoQ$ is some constant depending only on $d$ (note that $\deg(x)=2$ in the cohomological grading).

The cohomological integrality isomorphism is
\[
 \begin{matrix}
  \CP_{\lambda_0}\oplus\BoQ[x]^{\varepsilon_{d}}&\rightarrow&\BoQ[x^2]\\
  (f,g)&\mapsto&f+2C_{d}x^{\lfloor\frac{d}{2}\rfloor-1}g
 \end{matrix}
\]
where $\BoQ[x]^{\varepsilon_{d}}=\BoQ[x^2]$ if $\lfloor\frac{d}{2}\rfloor-1$ is even, and $\BoQ[x]^{\varepsilon_{d}}=x\BoQ[x^2]$ if $\lfloor\frac{d}{2}\rfloor-1$ is odd.

\begin{remark}
 We note, for example in \S\ref{subsection:GL2} and also \S\ref{subsection:SL2C2}, that the integrality morphisms are not isomorphisms when we consider cohomology with integral coefficients instead of rational coefficients. In \S\ref{subsection:GL2}, this comes from the presence of the factor $2$ in the cohomological integrality morphism.
\end{remark}

\subsection{$\rmSL_2(\BoC)\curvearrowright(\mathfrak{sl}_2(\BoC))^g$}
Let $g\geq 0$. We consider the diagonal adjoint action of $\rmSL_2(\BoC)$ on $V\coloneqq(\mathfrak{sl}_2(\BoC))^g$. Again, $\SP_V=\{\overline{\lambda_0},\overline{\lambda_1}\}$ with $\lambda_0,\lambda_1$ as in \S\ref{subsection:SL2C2}. The induction map reads
\[
 \begin{matrix}
  \Ind_{\lambda_1,\lambda_0}&\colon&\BoQ[x]&\rightarrow&\BoQ[x^2]\\
  &&f&\mapsto&(2x)^{g-1}f(x)+(-2x)^{g-1}f(-x)\,.
 \end{matrix}
\]
We have $V^{\lambda_1}\cong\BoC^g$ with the trivial action of $T$. Therefore, $\CP_{\lambda_1}=\BoQ$ and $\CP_{\lambda_0}=\BoQ[x^2]_{\deg<2(g-1)}$. The cohomological integrality isomorphism is
\[
 \begin{matrix}
  \BoQ[x]^{\varepsilon_g}\oplus\CP_0&\rightarrow&\BoQ[x^2]\\
  (f,g)&\mapsto&g(x)+2^gx^{g-1}f(x)
 \end{matrix}
\]
where $\BoQ[x]^{\varepsilon_g}=\BoQ[x^2]$ if $g-1$ is even and $\BoQ[x]^{\varepsilon_g}=x\BoQ[x^2]$ is $g-1$ is odd.
\begin{remark}
In this example, $G_{\lambda_0}=\{\pm1\}$ is a finite group. We see that it does not affect the cohomological integrality isomorphism.
\end{remark}

\subsection{$G\curvearrowright \{0\}$}
In this section, we present the cohomological integrality isomorphism for the trivial representation of a reductive group $G$. Then, $\SP_{\{0\}}/W$ is in bijection with the set of conjugacy classes of Levi subgroups of $G$. If $\lambda\in X_*(T)$ is a general cocharacter, then the induction morphism
\[
 \Ind_{\lambda}\colon \HO^*(\pt/T)\rightarrow\HO^*(\pt/G)
\]
is surjective. Indeed, we have
\[
 \Ind_{\lambda}(f)=\sum_{w\in W}w\left(\frac{f}{\prod_{\substack{\alpha\in\CW(\Fg)\\\langle\lambda,\alpha\rangle<0}}\alpha}\right).
\]
Therefore, $f=\Ind_{\lambda}\left(\frac{1}{\# W}f\prod_{\substack{\alpha\in\CW(\Fg)\\\langle\lambda,\alpha\rangle<0}}\alpha\right)$.

This implies that $\CP_{\lambda}=\{0\}$ if $G^{\lambda}$ is not the maximal torus of $G$ and $\CP_{\lambda}=\BoQ$ if $G^{\lambda}=T$. In particular, the cohomological integrality reads (for a choice of generic $\lambda\in X_*(T)$):
\[
 \begin{matrix}
  \BoQ[\Ft]^{\varepsilon_W}&\rightarrow&\BoQ[\Ft]^W\\
  f&\mapsto& \frac{f}{\prod_{\substack{\alpha\in\CW(\Fg)\\\langle\lambda,\alpha\rangle<0}}\alpha}\,,
 \end{matrix}
\]
where $\varepsilon_W(w)=(-1)^{\ell(w)}$ is the sign character of the Coxeter group $W$; $\ell(w)$ denotes the length of $w\in W$. This morphism is indeed an isomorphism since $\prod_{\substack{\alpha\in\CW(\Fg)\\\langle\lambda,\alpha\rangle<0}}\alpha$ is a generator of the $\BoQ[\Ft]^W$-module of anti-invariant polynomials.

\subsection{$G\curvearrowright \Fg$}
Let $G$ be a reductive group and $\Fg$ its Lie algebra with the adjoint action. Then, $\SP_{\Fg}$ is in one-to-one correspondence with the set of standard Levi subgroups of $G$ and $\SP_{\Fg}/W$ is in bijection with conjugacy classes of Levi subgroups of $\Fg$. In this case, $\BoQ[\Ft]^W$ is a polynomial algebra and for $\lambda\preceq\mu \in X_*(T)$, the induction map is
\[
 \begin{matrix}
  \Ind_{\lambda,\mu}&\colon&\BoQ[\Ft]^{W^{\lambda}}&\rightarrow &\BoQ[\Ft]^{W^{\mu}}\\
  &&f&\mapsto&\sum_{w\in W^{\mu}/W^{\lambda}}w\cdot f\,.
 \end{matrix}
\]
Therefore, the induction maps are surjective. This implies that $\CP_{\lambda}=0$ except if $\lambda$ is a generic cocharacter (i.e. $\Fg^{\lambda}$ is a torus), in which case $\CP_{\lambda}=\BoQ$. The cohomological integrality isomorphism is then
\[
 \begin{matrix}
  \BoQ[\Ft]^W&\rightarrow&\BoQ[\Ft]^W\\
  f&\mapsto&(\#W)\cdot f\,.
 \end{matrix}
\]

\subsection{Topology of the algebra of invariants}
Computing the polynomial invariants of binary forms is a long-standing problem in invariant theory, solved for binary forms with small degrees, and still challenging in higher degrees.

The strong version of the sheafified upgrade of our cohomological integrality theorem, which is the subject of a forthcoming work, would give an algorithm to compute the intersection cohomology of the GIT quotients $V\cms G$ for symmetric representations $V$ of a reductive group $G$. 
\begin{conjecture}[Strong integrality conjecture]
\label{conjecture:strongintegrality}
 Let $\lambda\in X_*(T)$. Then, there is a canonical identification $\CP_{\lambda}\cong \IH^*(V^{\lambda}\cms G^{\lambda})$ if a general closed orbit of $G^{\lambda}/G_{\lambda}$ inside $V^{\lambda}$ has finite stabilizer in $G^{\lambda}/G_{\lambda}$, and $\CP_{\lambda}=0$ otherwise.
\end{conjecture}

In particular, for $G=\rmSL_2(\BoC)$, it gives a computation of the intersection cohomology of $V_{d}\cms \rmSL_2$ where $V_d$ is the $d$-dimensional irreducible representation of $\rmSL_2(\BoC)$. In this section, we verify this strong integrality conjecture for some representations of $\rmSL_2(\BoC)$. More precisely, we check that for small values of $d$, when a general closed orbit in $V_d$ has finite stabilizers, then $\dim \IH(V_d\cms \rmSL_2(\BoC))=\dim \CP_0$, where $\CP_0$ is the vector space given by Theorem~\ref{theorem:cohintabsolute} for the trivial cocharacter $\lambda=0=\lambda_0$.

In our situation, motivated by integrality results, we also want to determine the \emph{stable} locus inside $V_{d}$, that is the open subset of closed $\rmSL_2(\BoC)$-orbits with finite stabilizer.

\begin{remark}
 While this paper was under refereeing process, the work \cite{bu2025cohomology} appears. The authors solve Conjecture~\ref{conjecture:strongintegrality} under an orthogonality assumption for the representation $V$ of $G$: they assume that there is a nondegenerate symmetric $G$-invariant bilinear on $V$. This condition is stronger than symmetricity (and a fortiori, weak symmetricity) we assume in the present paper, as demonstrated for example by the natural representation $\BoC^{2n}$ of the symplectic group $\mathrm{Sp}(2n)$. This example can be treated by a direct approach. In combination with \cite{knop2007invariant}, it is possible to prove Conjecture~\ref{conjecture:strongintegrality} for any symmetric representation $V$ of $G$ such that $\BoC[V]^G=\BoC$, that is $V\cms G=\pt$.
\end{remark}

\subsubsection{Invariants of linear binary forms}
Let $V_2=\BoC^2$ be the natural representation of $\rmSL_2(\BoC)$. Then, we calculate easily $V_2\cms\rmSL_2(\BoC)=\pt$ since the action of $\rmSL_2(\BoC)$ on $V_2$ has two orbits, $\{0\}$ and $V_2\setminus\{0\}$, of those only $\{0\}$ is closed. The stabilizer of the closed orbit $\{0\}$ is $\rmSL_2(\BoC)$, which is not finite. Therefore, the stable locus is empty. One expects that $\CP_{\lambda_0}=0$. This is confirmed by \S\ref{subsection:SL2C2}.

\subsubsection{Invariants of binary quadratic forms}
We let $V_3$ be the $3$-dimensional irreducible representation of $\rmSL_2(\BoC)$. There is  single polynomial invariant for binary quadratic forms $ax^2+bxy+cy^2$, which is the discriminant
\[
 D=b^2-4ac\,,
\]
see \cite[Satz~1.9, 2.8]{schur2013vorlesungen}. Therefore, we have $V_2\cms\rmSL_2(\BoC)\cong\BoC$. Since $\dim V_3/\rmSL_2(\BoC)=0$ and $\dim\BoC=1$, the stable locus is empty. Therefore, we expect that $\CP_{\lambda_0}=0$, and this is confirmed by \S\ref{subsection:SL2Cd}.

\subsubsection{Invariants of binary cubic forms}
We let $V_4$ be the $4$-dimensional irreducible representation of $\rmSL_2(\BoC)$. There is one polynomial invariant for binary cubic forms $ax^3+3bx^2y+3cxy^2+dy^3$, which is the discriminant
\[
 D=3b^2c^2+6abcd-4b^3d-4c^3a-a^2d^2\,,
\]
see \cite[Satz~2.8]{schur2013vorlesungen}. Therefore, $V_4\cms\rmSL_2(\BoC)\cong\BoC$. Moreover, since $\dim V_4/\rmSL_2(\BoC)=1=\dim\BoC$, the stable locus is non-empty. By $\BoC^*$-equivariance, the stable locus is the open locus of binary cubic forms of non-zero discriminant.

We notice that $\dim\IH(\BoC)=1=\dim\CP_{\lambda_0}=1$ where $\CP_{\lambda_0}$ is given in \S\ref{subsection:SL2Cd} for $d=4$.

\subsubsection{Invariants of binary quartic forms}
Let $V_5$ be the irreducible $5$-dimensional representation of $\rmSL_2(\BoC)$. The ring of invariants of binary quartic forms has two algebraically independent generators $i,j$ \cite[Satz~2.9]{schur2013vorlesungen}. Therefore, $V_5\cms\rmSL_2(\BoC)\cong\BoC^2$. Moreover, $\dim V_5/\rmSL_2(\BoC)=2=\dim\BoC^2$ and therefore, the stable locus is non-empty. In this case, we expect that $P_{\lambda_0}$ is of dimension $1=\dim\IH^2(\BoC^2)$, and this is again confirmed by \S\ref{subsection:SL2Cd} for $d=5$.

\section*{Competing interests}
The author has no competing interest to declare.
\printbibliography
\end{document}